\documentclass[a4paper, 11pt]{article}
\usepackage{anysize}
\marginsize{3,5cm}{2,5cm}{2,5cm}{2,5cm}%tutaj ustawiasz marginesy
\usepackage{amssymb}
\usepackage{amsmath,amsbsy,amssymb,amscd}
\usepackage{t1enc}
\usepackage[cp1250]{inputenc}
\usepackage[british]{babel}
\usepackage[all]{xy}
\usepackage{color}
\usepackage{amsfonts}
\usepackage{latexsym}
\usepackage{amsthm}
\usepackage{mathrsfs}
\usepackage{hyperref}

\usepackage{color}
\usepackage{enumitem}

\usepackage{changes}

\definecolor{orange}{rgb}{1,0.5,0}

\usepackage{mathrsfs} %Script dla zapisu \mathscr{}

\DeclareMathAlphabet{\mathpzc}{OT1}{pzc}{L}{it} %stylizowany
\newtheorem{definition}{Definition}[section]

\newtheorem{proposition}[definition]{Proposition}
\newtheorem{theorem}{Theorem}

\newtheorem{remark}[definition]{Remark}
\newtheorem{lemma}[definition]{Lemma}

\def\C{\mathbb{C}}
\def\geq{\geqslant}
\def\leq{\leqslant}
\def\R{\mathbb{R}}

\def\Z{\mathbb{Z}}
\def\N{\mathbb{N}}

\def\la{\langle}
\def\ra{\rangle}

\def\cB{\mathcal{B}}

\newcommand{\mf}{\mathfrak}
\newcommand{\bea}{\begin{eqnarray}}
  \newcommand{\eea}{\end{eqnarray}}
  \newcommand{\beab}{\begin{eqnarray*}}
  \newcommand{\eeab}{\end{eqnarray*}}

  \newcommand{\be}{\begin{equation}}
  \newcommand{\ee}{\end{equation}}

\newcommand*{\diff}{\mathop{}\!\mathrm{d}}
\newcommand{\norm}[1]{\left\lVert#1\right\rVert}
\newcommand{\SL}{%
\operatorname{SL}
}

\newcommand{\Spec}{%
\operatorname{Spec}
}

\title{Polynomial 3-mixing for smooth time-changes of horocycle flows}
\author{Adam Kanigowski and Davide Ravotti}
\begin{document}
\maketitle
\begin{abstract} Let $(h_t)_{t\in \R}$ be the horocycle flow acting on $(M,\mu)=(\Gamma \backslash \SL(2,\R),\mu)$, where $\Gamma$ is a co-compact lattice in $\SL(2,\R)$ and $\mu$ is the homogeneous probability measure locally given by the Haar measure on $\SL(2,\R)$. Let $\tau\in W^6(M)$ be a strictly positive function and let $\mu^{\tau}$ be the measure equivalent to $\mu$ with density $\tau$. We consider the time changed flow $(h_t^\tau)_{t\in \R}$ and we show that there exists $\gamma=\gamma(M,\tau)>0$ and a constant $C>0$ such that for any $ f_0, f_1, f_2\in W^6(M)$ and for all $0=t_0<t_1<t_2$, we have
$$
\left|\int_M \prod_{i=0}^{2} f_i\circ h^\tau_{t_i} \diff \mu^\tau -\prod_{i=0}^{2}\int_M f_i \diff \mu^\tau \right|\leq C \left(\prod_{i=0}^{2} \|f_i\|_6\right) \left(\min_{0\leq i<j\leq 2} |t_i-t_j|\right)^{-\gamma}.
$$
With the same techniques, we establish polynomial mixing of all orders under the additional assumption of $\tau$ being fully supported on the discrete series.
\end{abstract}

\section{Introduction}

\subsection{Unipotent flows and their time-changes}

Unipotent flows on compact (or, in general, finite volume) quotients of Lie groups are homogeneous flows given by the action of one-parameter unipotent subgroups.
An important example of a unipotent flow is the \emph{horocycle flow} on compact quotients $\Gamma \backslash \SL(2,\R)$ of $\SL(2,\R)$, defined by multiplication on the right by $\begin{pmatrix}1&t\\0&1\end{pmatrix}$.
Identifying $\Gamma \backslash \SL(2,\R)$ with the unit tangent bundle of the compact hyperbolic surface $S=\Gamma \backslash \mathbb{H}$, the horocycle flow is the unit speed parametrization of translations along the stable leaves of the geodesic flow on $T^1S$. 

Dynamical properties of horocycle flows have been studied in great details and are now well-understood: they have zero entropy \cite{Gur}, in the compact setting are minimal \cite{Hed}, uniquely ergodic \cite{Fur}, mixing and mixing of all orders \cite{Mar}, and have countable Lebesgue spectrum \cite{Par} (mixing and spectral properties hold for general finite volume quotients). Finer ergodic properties were investigated by Ratner \cite{Rat2, Rat3}.

Another important class of unipotent flows is given by \emph{nilflows on nilmanifolds}, namely homogeneous flows on compact quotients of (non-abelian) nilpotent Lie groups. The prototypical examples of nilflows are \emph{Heisenberg nilflows} on quotients of the 3-dimensional Heisenberg group.

One key feature of unipotent flows, in particular of the horocycle flow, is a form of \emph{slow divergence}: the distance between nearby points lying on different orbits grows at most polynomially in time (quadratically, in the case of horocycle flows). This property is in sharp contrast with the dynamics of \emph{hyperbolic} flows, such as the geodesic flow, for which the divergence of orbits is exponential. 
Unipotent flows are hence examples of smooth \emph{parabolic} flows, namely smooth flows for which nearby points diverge polynomially in time. 

Outside the homogeneous setting, very little is known for general smooth parabolic flows, even for smooth perturbations of homogeneous ones.
Perhaps the simplest case of such perturbations are smooth \emph{time-changes}, or \emph{time-reparametrizations}. Roughly speaking, a smooth time-change of a flow is obtained by moving along the same orbits, but varying smoothly the speed of the points. 
In other words, a smooth time-change is defined by rescaling the generating vector field by a smooth function $\tau$, called the \emph{generator} of the time-change, see Section \ref{sec:tch} for definitions.
A time-change is said to be \emph{trivial} if its generator is a \emph{quasi-coboundary} for the flow, see Section \ref{sec:tch}. It is easy to see that trivial time-changes are isomorphic to the original flow.

On the other hand, performing a non-trivial smooth time-change can alter significantly the ergodic properties of the flow. 
This is the case, for example, of ergodic nilflows. 
Indeed, nilflows are never weakly mixing, because of the presence of a toral factor, corresponding to the projection onto the abelianization of the nilpotent group. Nevertheless, non trivial time-changes, within a natural class of \lq\lq polynomial\rq\rq\ functions on the nilmanifold,  destroy the toral factor and are strongly mixing, as was shown by Avila, Forni, Ulcigrai, and the second author in \cite{AFRU}, extending previous results in \cite{AFU} and in \cite{Rav}. For time-changes of bounded type Heisenberg nilflows, one obtains an even stronger dichotomy, \cite{FK}: either the time-change is trivial (in which case the toral factor-persists), or the time-changed flows is \emph{mildly mixing} (it has no non-trivial rigid factors).

In the case of the horocycle flow, the study of the cohomological equation by Flaminio and Forni \cite{FF} imply that a generic time-change of the horocycle flow is non-trivial and thus, by the rigidity result of Ratner \cite{Rat5}, not even measurably conjugated to the horocycle flow itself. Hence, non-trivial time-changes form an important family of smooth \emph{non-homogeneous} parabolic flows.
Similarly to the unperturbed horocycle flow, they are mixing, as was shown by Marcus \cite{Mar2}. Moreover, it was conjectured by Katok and Thouvenot, \cite{Kat-Tho}, that sufficiently smooth time changes of horocycle flows have countable Lebesgue spectrum. Lebesgue spectral type for smooth time-changes was proved by Forni and Ulcigrai \cite{ForU} (independently, Tiedra de Aldecoa \cite{TdA} obtained the absolute continuity property). The full version of the Katok-Thouvenot conjecture (countable multiplicity) was recently obtained in \cite{FFK}.

However, as it happens for nilflows, other finer properties of non trivial time-changes are different from their homogeneous counterpart. One such example is the set of joinings between their rescalings: whilst all rescalings of the horocycle flow are isomorphic to each other, the first author, Lema\'nczyk and Ulcigrai \cite{KLU}, and Flaminio and Forni \cite{FF2} independently, showed that different rescalings of non-trivial time-changes are always disjoint.

\subsection{Quantitative mixing}

Let $k \in \N$, $k \geq 2$. We recall that a measure preserving flow $\{\varphi_t \colon M \to M\}_{t \in \R}$ on a probability space $(M, \mathscr{B}, \mu)$ is said to be \emph{$k$-mixing} if for any $f_0, \dots, f_{k-1} \in L^{\infty}(M)$ we have
\begin{equation}\label{eq:k_mixing}
\left\lvert \int_M f_0 \cdot f_1 \circ \varphi_{t_1} \cdots f_{k-1}\circ \varphi_{t_{k-1}} \diff \mu - \left( \int_M f_0 \diff \mu \right) \cdots \left( \int_M f_{k-1} \diff \mu \right) \right\rvert \to 0
\end{equation}
as $|t_i - t_j| \to \infty$, for all $t_i \neq t_j$, $0 \leq i,j \leq k-1$. 
We say that $\varphi_t$ is \emph{mixing of all orders} if it is $k$-mixing for all $k \geq 2$.
In the case of the horocycle flow, it follows from \cite{Rat4} that Ratner's property persists under smooth time-changes, hence all smooth time changes of the horocycle flow are mixing of all orders.

Under some regularity assumptions on the observables $f_i$, one can ask about the \emph{rate of decay} in the limit \eqref{eq:k_mixing} in terms of the minimum $|t_i - t_j|$ for $i \neq j$. It turns out that in the parabolic setting, quantitative $2$-mixing is more tractable than quantitative higher order mixing as we describe below.

\subsubsection{Quantitative 2-mixing}
For parabolic flows (i.e. flows of intermediate orbit growth), quantitative 2-mixing is in most cases based on controlled (quantitative) stretching of certain curves by the flow.  Ratner, \cite{Rat1}
%, used stretching of geodesic curves to 
proved that the rate of 2-mixing of the horocycle flow is \emph{polynomial}, namely she showed that there exists an explicit $\gamma > 0$, depending only on the co-compact lattice $\Gamma$, such that for all $\mathscr{C}^3$ functions $f_0,f_1$ there exists a constant $C = C(f_0, f_1)$ such that for all $t \geq 1$ we have 
$$
\left\lvert \int_M f_0 \cdot f_1 \circ \varphi_{t} \diff \mu - \left( \int_M f_0 \diff \mu \right) \left( \int_M f_{1} \diff \mu \right)\right\rvert \leq C t^{-\gamma}.
$$
Moreover, it can be shown that this bound is optimal. In the case of time-changes of horocycle flows, quantitative mixing estimates were obtained by Forni and Ulcigrai in \cite{ForU}, although they are conjecturally not optimal. Their result is based on sharp bounds on ergodic integrals of the horocycle flow proved by Flaminio and Forni in \cite{FF} and refined by Bufetov and Forni for \lq\lq horocycle-like\rq\rq\ arcs in \cite{BuFo}, together with stretching of geodesic curves. %(as in \cite{Rat1}).

For other parabolic flows, Forni and the first author in \cite{FK2} showed that, for a full dimensional set of Heisenberg nilflows and for a generic set of smooth time-changes, if the time-change is not trivial, the rate of mixing is polynomial. This is the only quantitative result available for mixing properties of time-changes of nilflows. 

A shearing phenomenon analogous to the one described above is at the base of several results on quantitative 2-mixing for non-homogeneous parabolic flows, see e.g., \cite{Fay}, \cite{Rav}, \cite{FFK}. We will  use a version of this mechanism in this paper as well, see the proof of Theorems \ref{thm:main3} and \ref{main:th2} in Sections \ref{sec:3mixproof} and \ref{sec:proof_of_Theorem_2}.

\subsubsection{Quantitative higher order mixing}
Quantitative higher order mixing (in particular, $3$-mixing) for parabolic flows is much harder to get and, until recently, there were no results in the literature on this problem. The main reason for this is that mechanisms for obtaining higher order mixing are, by their very nature, non-quantitative: singular spectrum criterion of Host \cite{Ho}, Ratner's property \cite{Rat3}, or Marcus multiple mixing mechanism \cite{Mar}.

The first, and to the best of our knowledge the only, quantitative higher order mixing result for parabolic systems appears in the very recent work of Bj\"orklund, Einsiedler, and Gorodnik \cite{BEG}, where the authors proved a very general quantitative result for multiple mixing of group actions which, in the very specific case of the regular action of $\SL(2,\R)$, implies that, for all $k\geq 2$, the rate of $k$-mixing of the horocycle flow is \emph{polynomial}. Such results are difficult to obtain for non-homogeneous flows, and in particular for non trivial time-changes of unipotent flows, since one cannot exploit the algebraic properties of the actions and the powerful representation theory machinery.

\subsection{Statement of the main results}

In this paper, we establish \emph{polynomial $3$-mixing} estimates for any smooth time-change of the horocycle flow, see Theorem \ref{thm:main3} below.
To the best of our knowledge, this is the first quantitative mixing result beyond 2-mixing for smooth non-homogeneous parabolic flows. 

Let  $(h_t)_{t\in \R}$ be the horocycle flow on $(M,\mu)$, where $M=\Gamma \backslash \SL(2,\R)$ is compact and $\mu$ is locally given by the Haar measure. Let $W^6(M) \subset  L^2(M)$ denote the standard Sobolev space or order $6$ (see Section \ref{sec:spec} for definitions), and let $\tau \in W^6(M)$ be a positive function. We consider the time changed flow $(h_t^\tau)_{t\in \R}$ generated by $\tau$ as defined in  Section \ref{sec:tch}. 
The following is our main result.

\begin{theorem}\label{thm:main3}
Let $\tau \in W^6(M)$ be a positive function. There exists $\gamma=\gamma(M,\tau)>0$ and a constant $C>0$ such that for any 
$ f_0, f_1, f_2 \in W^6(M)$ and for all $0=t_0<t_1<t_2$, we have 
$$
\left|\int_M f_0 \cdot (f_1\circ h^\tau_{t_1}) \cdot (f_2\circ h^\tau_{t_2})\diff \mu^\tau -\prod_{i=0}^{2}\int_M f_i \diff \mu^\tau \right|\leq C \left(\prod_{i=0}^{2} \|f_i\|_6\right) \left(\min_{0\leq i<j\leq 2} |t_i-t_j|\right)^{-\gamma}.
$$
\end{theorem}

With the same techniques, we are able to prove polynomial mixing \emph{of all orders} only for time-changes supported on the \emph{discrete series} $\mathcal{H}_d$ (see Section \ref{sec:spec} for definitions). The proof however present several additional technical difficulties compared to the 3-mixing case, hence we present it in the Appendix \ref{sec:appendix}.

\begin{theorem}\label{thm:main}
Let $\tau \in W^6(M) \cap \mathcal{H}_d$ be a positive function, and let $k\in \N$. There exists $\gamma=\gamma(M,k,\tau)>0$ such that for any 
$ f_0,\ldots,f_{k-1}\in W^6(M)$ there exists $C=C(f_0,\ldots,f_{k-1})>0$, such that for all $0=t_0<t_1<\ldots<t_{k-1}$, we have 
$$
\left|\int_M \prod_{i=0}^{k-1} f_i\circ h^\tau_{t_i} \diff \mu-\prod_{i=0}^{k-1}\int_M f_i \diff \mu\right|\leq C\left(\min_{0\leq i<j\leq k-1} |t_i-t_j|\right)^{-\gamma}.
$$
\end{theorem}
%\textcolor{blue}{In the theorem above, shouldn't we have $(t_i)$ increasing and the difference of two consecutive times large?} \textcolor{red}{I actually don't see any point in which we need it to be large, so probably we don't need this assumption? Maybe just to be larger than 1 for the bound to be meaningful (and to apply Lemma 2.2)?}

The driving idea of the proof is refining Marcus' approach for multiple mixing of the horocycle flow in \cite{Mar} by making it quantitative. 
Our argument shares some similarities with the one in \cite{BEG}, notably in exploiting the shearing of a transverse vector field under the action (see in particular \cite[\S7.2]{BEG}). For homogeneous flows, the push-forward of left-invariant vector fields is given by the Adjoint, which can be controlled using the algebraic structure of the group, see \cite[\S2]{BEG}. In our setting, however, due to the non-homogeneous structure of the flow, we employ a more geometric approach and we exploit precise bounds on the growth of ergodic integrals and good quantitative control of the (\emph{non-uniform}) stretching of geodesic curves.
%An important difference with Marcus' result, however, is that the generic time-change considered here is not homogeneous and, in particular, the shearing of geodesic segments under the action of the flow is \emph{not uniform}. 
In the proofs of Theorem \ref{thm:main3} and Theorem \ref{thm:main}, the problem is reduced to study the $L^2$ norm of some multiple ergodic averages, see Propositions \ref{lem:3mixcase} (and the more general version in Proposition \ref{lem:l2}), which are estimated using a sharp quantitative version of van der Corput inequality (Lemma \ref{lem:vdc}).
We hope that the local mechanism that we use has the potential to be applied to other non-homogeneous flows, such as time-changes of higher step nilflows, or some smooth surface flows.

We should emphasize that at this moment we do not know how to generalize Theorem \ref{thm:main} to higher order correlations (for functions having non-trivial support outside the discrete series). The main reason is that in the case of $3$-mixing we face one of the two situations: either $t_1$ and $t_2$ are of similar order (in which case it is possible to apply Proposition \ref{lem:3mixcase}) or $t_1$ is much smaller than $t_2$ (in which case we use the fact that appropriate length geodesic segments are not stretched for time $t_1$, whereas they stretch for time $t_2-t_1$ and we use invariance of measure). The reader will notice that in both cases the choice of the length $\sigma$ of the geodesic segments is rather delicate. This mechanism seems not to work even for the case of $4$-mixing especially in the case if $t_1$ is much smaller than $t_3$ and of order $t_3-t_2$: on one hand, a meaningful estimate using Proposition \ref{lem:l2} would force $\sigma$ to be larger than $(t_3 - t_2)^{-1}$, on the other hand, controlling the deviations from the homogeneous case requires $\sigma$ to be smaller than some negative power of $t_3$, and hence an appropriate choice of $\sigma$ is not possible. We can handle this problem assuming additionally that the time change $\tau$ is fully supported on the disrecte series (Theorem \ref{thm:main}): in this case the deviation of ergodic averages for $\tau$ are logarithmic (see Lemma \ref{lem:erren}). We then inductively get polynomial $k+1$-mixing from polynomial $k$-mixing (using logarithmic deviation bounds for the time change).

\section{Definitions and basic properties}
\subsection{Time changes of flows}\label{sec:tch}
Let $(\varphi_t)$ be a flow on $(X,\cB,\mu)$ and let $\tau\in L^1(X,\mu)$ be a strictly positive function. Then, for a.e.~$x\in X$, for every $t\in \R$, there exists a unique solution $u=u(x,t)$ of
$$
\int_{0}^u \tau(\varphi_sx) \diff s=t.
$$
The function $u(x,\cdot)$ defined this way is an $\R$-cocycle, i.e.~for $t_1,t_2\in \R$, we have $u(x,t_1+t_2)=u(x,t_1)+u(\varphi_{t_1}x,t_2)$. We define the \emph{time-change} flow $(\varphi_t^\tau)_{t\in \R}$ induced by $\tau$ by setting $\varphi_t^\tau(x)=\varphi_{u(x,t)}(x)$, and we say that $\tau$ is its \emph{generator}. Since $u(x,\cdot)$ is a cocycle, the latter equality defines an $\R$-action. Moreover,  $(\varphi_t^\tau)_{t\in \R}$ preserves the measure $\mu^\tau$ given by $\diff \mu^\tau=\frac{\tau}{\int_X \tau \diff \mu} \diff \mu$. We will always WLOG assume that $\int_X \tau \diff \mu = 1$.

Since the flow $(\varphi_t)$ has the same orbits as any of its time-changes, and since the invariant measure $\mu^\tau$ is equivalent to $\mu$, ergodicity is preserved when performing a smooth time-change. Mixing and other spectral properties, however, are more delicate, as discussed in the introduction.

We say that a function $\tau$ is a {\em quasi-coboundary} for $(\varphi_t)_{t\in \R}$ if there exists a measurable solution $\xi \colon X\to \R$ to 
$$
\int_0^t\tau(\varphi_sx)ds-t \int_X \tau d\mu= \xi(x)-\xi(\varphi_tx), \text{ for } t\in \R.
$$
It follows that if $\tau$ is quasi-coboundary, then $(\varphi_t)_{t\in \R}$ and $(\varphi_t^\tau)_{t\in \R}$ are {\em isomorphic}. We call such time changes \emph{trivial}.

\subsection{Horocycle and geodesic flows}
Let $G:=\SL(2,\R)$ be the group of $2\times 2$ matrices with determinant $1$ and let $\mu$ be the Haar measure on $G$. We denote the lie algebra of $G$ by $\mf{g}$, which consists of $2\times 2$ matrices of zero trace. Let $U,X,V\in \mf{g}$ be given by 
$$
U:=\begin{pmatrix}0&1\\0&0\end{pmatrix},\;\;\;\; X:=\begin{pmatrix}1/2&0\\0&-1/2\end{pmatrix},\;\;\;\;V:=\begin{pmatrix}0&0\\1&0\end{pmatrix}.
$$
Then $U, X,V$ are  generators of respectively the (stable) horocycle, geodesic and opposite (unstable) horocycle flow. We will be dealing with flows generated by $U$ and $X$. More precisely, let $\exp \colon \mf{g}\to G$ be the exponential map and let $\Gamma\subset G$ be a co-compact lattice in $G$. We will consider  the following $\R$-actions on the homogeneous space $M:=\Gamma \backslash G$: the horocycle flow
\be\label{eq:hor}
h_t(\Gamma x)=\Gamma x \exp(tU), 
\ee
and the geodesic flow
\be\label{eq:geo}
g_t(\Gamma x)=\Gamma x \exp(tX).
\ee
The flows $h_t$ and $g_t$ both preserve a smooth measure on $M$, locally given by the Haar measure $\mu$, which we will denote also by $\mu$. Recall that the horocycle and geodesic flows satisfy the following {\em renormalization} equation 
\be\label{eq:ren}
h_t \circ g_s=g_s \circ h_{e^{s}t}, \text{ for every } t,s\in \R.
\ee
\subsection{Spectral theory of horocycle flows}\label{sec:spec}
We will briefly recall some facts from the spectral theory of horocycle flows, for details see e.g.~\cite{FF}. Let 
$$
\Theta:=\begin{pmatrix}0&1/2\\-1/2&0\end{pmatrix}
$$
be the generator of the maximal compact subgroup $\text{SO(2)}$ of $G=\SL(2,\R)$.
Let $\mathcal{H} = L^2(M,\mu)$ be the Hilbert space of square integrable functions on $M = \Gamma \backslash G$, on which $G$ acts unitarily. 
We define the Laplacian by setting $\Delta:=-(X^2+U^2/2+V^2/2)$; it is an \emph{elliptic} element of the universal enveloping algebra of $\mathfrak{g}$ which acts as an essentially self-adjoint operator on $\mathcal{H}$. Remark that $\Delta$ on $\text{SO(2)}$-invariant functions coincides with the Laplace-Beltrami operator on the compact hyperbolic surface $S= \Gamma \backslash \mathbb{H}$. The {\em Sobolev space} of order $s>0$, $W^s(M)$, is defined as the completion of the space $C^\infty(M)$ of infinitely differentiable functions with respect to the inner product 
$$
\la f,g \ra:=\la (1+\Delta)^sf,g \ra_{\mathcal{H}}. 
$$
We will denote by $\norm{\cdot}_6$ the norm in $W^6(M)$.

Let $\square:=-X^2-(V+\Theta)^2+\Theta^2 = \Delta + 2 \Theta^2$ be the \emph{Casimir operator}, a generator of the centre of the universal enveloping algebra of $\mathfrak{g}$. By the classical theory of unitary representations of $\SL(2,\R)$, we have the following orthogonal decomposition into irreducible components, listed with multiplicity:
$$
L^2(M)=\bigoplus_{\mu\in \Spec(\square)}H_\mu= \mathcal{H}_{p}\oplus \mathcal{H}_{c}\oplus \mathcal{H}_{d},
$$
where
$$ 
\mathcal{H}_{p}=\bigoplus_{\begin{subarray}{c}\mu\in \Spec(\square),\\ \mu\geq 1/4\end{subarray}} H_\mu,\qquad \mathcal{H}_{c}=\bigoplus_{\begin{subarray}{c}\mu\in \Spec(\square),\\ \mu\in(0,1/4)\end{subarray}} H_\mu,\qquad \mathcal{H}_{d}=\bigoplus_{\begin{subarray}{c}\mu\in \Spec(\square),\\ \mu=-n^2+n,\,n\in\Z_{\geq 0}\end{subarray}} H_\mu.
$$
The decomposition above induces a corresponding decomposition of the Sobolev spaces $W^r(M)$, for all $r>0$.

We call $\mathcal{H}_{p}$ the \emph{principal series}, $\mathcal{H}_{c}$ the \emph{complementary series}, and $\mathcal{H}_{d}$ the \emph{discrete series}. 
On each irreducible representation $H_{\mu}$, the Casimir operator acts as multiplication by the constant $\mu$. 
The representation $H_0$ is the trivial representation and appears with multiplicity 1.
We recall that the positive eigenvalues $\mu$ of the Casimir operator coincide with the eigenvalues of the Laplace-Beltrami operator on the surface $S = \Gamma \backslash \mathbb{H}$, in particular there is a \emph{spectral gap}: there exists $\mu_0 >0$ such that $(0,\mu_0) \cap \Spec(\square) = \emptyset$. 
Let us further define
$$
\nu_0 :=
\begin{cases}
\sqrt{1-4\mu_0} & \text{ if } \mu_0 < 1/4,\\
0 & \text{ if } \mu_0 \geq 1/4,
\end{cases}
\text{\ \ \ and\ \ \ }
\varepsilon_0 :=
\begin{cases}
0 & \text{ if } \mu_0 \neq 1/4,\\
1 & \text{ if } \mu_0 = 1/4.
\end{cases}
$$

In the second part of the paper, Appendix \ref{sec:appendix}, we will be interested in functions $\tau \in \mathcal{H}_d$. 
We remark that it follows from a recent work of D.\ Dolgopyat and O.\ Sarig \cite{DoSa} that functions coming from non-zero harmonic forms are not measurable coboundaries; in particular, there exist positive functions in $\mathcal{H}_d$ which are not measurable quasi-coboundaries, and hence generate a time-change which is not measurably conjugate to the horocycle flow.

\subsection{2-mixing estimates for time changes of the horocycle flow}

Let us denote by $(h^\tau_t)$ the time change of the horocycle flow $(h_t)$ induced by the positive function $\tau$. We make a standing assumption that $\int_M \tau \diff\mu=1$.

We recall a result of G.~Forni and C.~Ulcigrai, \cite{ForU}, on estimates of rates of 2-mixing for time-changes of the  horocycle flow.
In the homogeneous setting, optimal rates of mixing for the classical horocycle flow were obtained by M.~Ratner in \cite{Rat1}.

\begin{lemma}[Theorem 3, \cite{ForU}]\label{lem:fu} 
Let $\tau \in W^6(M)$, $\tau>0$ and let $(h_t^\tau)$ denote the time change induced by $\tau$. There exists a constant $C_0=C_0(\tau)>0$ such that for any functions $f,g\in W^6(M)$ and any $t >1$, we have 
$$
\left|\int_M (f\circ h^\tau_t)g \diff \mu^{\tau}- \int_M f \diff \mu^{\tau} \int_Mg \diff \mu^{\tau}\right|\leq C_0 \|f\|_6\|g\|_6 t^{-\frac{1-\nu_0}{2}} (\log t)^{\varepsilon_0}.
$$
\end{lemma}

In order to prove Theorem 3 in \cite{ForU}, the authors establish the following lemma, which will be useful for us as well.
\begin{lemma}[Lemma 18, \cite{ForU}]\label{lem:fu18} 
Let $\tau \in W^6(M)$, $\tau>0$ and let $(h_t^\tau)$ denote the time change induced by $\tau$. There exist constants $\overline{\sigma}, C_1 >0$ such that for any function $f\in W^6(M) \cap L^2_0(M)$, any $x \in M$, any $0 <\sigma < \overline{\sigma}$, and any $t>1$, we have 
$$
\left|\frac{1}{\sigma} \int_0^\sigma  f\circ h^\tau_t \circ g_r (x) \diff r \right|\leq C_1 \|f\|_6  (\sigma t)^{-\frac{1-\nu_0}{2}} (\log (\sigma t))^{\varepsilon_0}.
$$
\end{lemma}
%\begin{remark}\label{rem:bet}
%It follows \cite{ForU} that the above results hold for any $0 < \beta < \frac{1-\nu_0}{2}$, or also $\beta =  \frac{1-\nu_0}{2}$ if $\mu_0 \neq \frac{1}{4}$.
%\end{remark}
Let us fix $0 < \beta < \frac{1-\nu_0}{2}$ (or, in the case $\mu_0 \neq 1/4$, one can choose $\beta = \frac{1-\nu_0}{2}$), so that, in particular, Lemma \ref{lem:fu} and Lemma \ref{lem:fu18} hold with a bound of the form $t^{-\beta}$ and $(\sigma t)^{-\beta}$ respectively.

\subsection{Deviation of ergodic averages}

We will first state a result  on the growth of ergodic integrals, which is a straightforward consequence of Theorem 1.5 in \cite{FF}. 

\begin{lemma}\label{lem:dev}
Let $\tau\in W^6(M)$. There exists a constant $C_2$ such that for every $0 < s <1$, every $T>1$ and every $x\in M$, we have 
$$
\left\lvert \int_{0}^T (\tau-\tau\circ g_s)(h_tx) \diff t \right\rvert \leq C_2 s T^{1-\beta}.
$$
Moreover, if $\tau\in W^6(M) \cap \mathcal{H}_d$, the integral above is bounded by $C_2 s \log T$.
\end{lemma}
\begin{proof} 
By Theorem 1.5 in \cite{FF}, we have 
$$
\left|\int_{0}^T (\tau-\tau\circ g_s)(h_tx) \diff t\right|\leq C \|\tau-\tau\circ g_s\|_6 T^{\frac{1+\nu_0}{2}}(\log T)^{\varepsilon_0}  \leq C C's T^{\frac{1+\nu_0}{2}}(\log T)^{\varepsilon_0},
$$
which finishes the proof by the choice of $\beta$ and by taking $C_2=CC'$.
If we further assume that $\tau$ belongs to the discrete series, the estimate follows again from Theorem 1.5 in \cite{FF}, after noticing that the space $\mathcal{H}_d$ is invariant for the action of the geodesic flow $g_s$, so that $\tau-\tau\circ g_s\in \mathcal{H}_d$ for any $0<s<1$. 
\end{proof}

We remark that, since any non-trivial time change destroys the homogeneous structure, the commutation relation \eqref{eq:ren} does not in general hold for time-changes. Below we state an important lemma which estimates the error in the renormalization formula for the time changed flow.

\begin{lemma}\label{lem:erren} 
Let $\tau\in W^6(M)$, $\tau >0$. There exists $C_3>0$ such that for every $x\in M$, $0<s<1$ and $T>1$, we have 
$$
h_T^\tau \circ g_s(x)=g_s \circ h_{e^{s}T+ A(x,s,T)}^\tau (x),
$$
where 
$$
|A(x,s,T)|\leq C_3 s T^{1-\beta}.
$$
Moreover, if $\tau\in W^6(M)\cap \mathcal{H}_d$, we have $|A(x,s,T)|\leq C_3 s\log T$.
\end{lemma}
\begin{proof} 
Let $A(x,s,T)$ be such that 
\be\label{eq:xst}
u(x,e^{s}T+ A(x,s,T))=e^{s}u(g_sx,T).
\ee
Notice that for every fixed $x \in M$, the function $u(x,\cdot)$ is strictly increasing, hence the term $A(x,s,T)$ as in \eqref{eq:xst} above is uniquely defined.
By definition of a time change and \eqref{eq:ren}, we have 
$$
h_T^\tau \circ g_s(x)=h_{u(g_sx,T)} \circ g_s(x)=g_s \circ h_{e^{s}u(g_sx,T)}(x)=g_sh_{e^{s}T+ A(x,s,T)}^\tau x.
$$
We only need to show that $A(x,s,T)$ given by \eqref{eq:xst} satisfies the desired estimate for some constant $C_3>0$.

By definition, we have 
$$
e^{s}T+A(x,s,T)=\int_{0}^{u(x,e^{s}T+A(x,s,T))}\tau(h_tx) \diff t.
$$
Changing variables $r=e^st$ and using \eqref{eq:ren}, we get
$$
e^{s}T=e^{s}\int_{0}^{u(g_sx,T)}\tau(h_t \circ g_s(x)) \diff t=\int_{0}^{e^{s}u(g_sx,T)}\tau(g_s \circ h_r(x)) \diff r.
$$
Therefore, \eqref{eq:xst} gives us 
$$
A(x,s,T)=\int_{0}^{e^{s}u(g_sx,T)}\tau(h_tx) \diff t-\int_{0}^{e^{s}u(g_sx,T)}\tau(g_sh_tx) \diff t.
$$
Using Lemma \ref{lem:dev}, we obtain
$$
|A(x,s,t)|=\left|\int_{0}^{e^{s}u(g_sx,T)}(\tau-\tau\circ g_s)(h_tx) \diff t\right|\leq C_2 s |e^{s}u(g_sx,T)|^{\frac{1+\nu_0}{2}} (\log |e^{s}u(g_sx,T)|)^{\varepsilon_0}.
$$
Since $\max_{x \in M} |u(x,T)| \leq \frac{1}{\inf_M \tau} T$, the proof is complete.
\end{proof}

\section{Van der Corput inequality}

We recall a version of the van der Corput's inequality, that will be useful in our setting. The following lemma is valid in general Hilbert spaces $H$, for simplicity we state it just for $H=L^2(X,\mu)$, where $(X,\mu)$ is a probability space. The notation $X= {\rm O}(Y)$ means that $X\leq cY$  for some global constant $c>0$.

\begin{lemma}[Van der Corput inequality]\label{lem:vdc} 
Let $(\phi_u)_{u\in \R}\subset L^2(X,\mu)$ with $\|\phi_u\|_{2}\leq 1$ for every $u\in \R$ and assume that $\la\phi_{u},\phi_w\ra=\la\phi_0,\phi_{w-u}\ra$ for every $u,w\in \R$. Then, for every $N>0$ and $0<L<N$, we have 
\be\label{eq:vdc}
\left\|\frac{1}{N}\int_{0}^N\phi_u \diff u\right\|_{2}\leq  \left[ \frac{2}{N} \int_0^N \left(\frac{1}{L}\int_0^L|\la \phi_u,\phi_{u+l}\ra|\diff l \right)\diff u\right]^{1/2}+{\rm O}\left(\frac{L}{N}\right)
\ee
\end{lemma}

\begin{remark} As mentioned before, the result above is true for general Hilbert spaces and without the extra invariance assumption on the $(\phi_u)_{u\in \R}$.  We will use Lemma \ref{lem:vdc} for $\phi_u=f\circ h^\tau_{u}$ for which the above assumption is satisfied. A nice proof of the more general statement in the non-quantitative version can be found in J.\ Moreira blogpost, \cite{JM}.
\end{remark}
The proof follows standard steps, we provide it here for completeness.
\begin{proof}[Proof of Lemma \ref{lem:vdc}] 
Notice first that 
$$
\left\|\frac{1}{N}\int_0^N\phi_u \diff u- \frac{1}{L}\int_0^L\left(\frac{1}{N}\int_0^N\phi_{u+l} \diff u \right) \diff l \right\|_{2}={\rm O}\left(\frac{L}{N}\right).
$$
Moreover, by Cauchy-Schwartz inequality, 
\begin{equation*}
\begin{split}
& \left\|\frac{1}{N}\int_0^N\left(\frac{1}{L}\int_0^L\phi_{u+l} \diff l \right)\diff u\right\|_{2} \leq \left[ \frac{1}{N} \int_0^N\left\|\frac{1}{L}\int_0^L\phi_{u+l} \diff l \right\|_2^2 \diff u\right]^{1/2}\\
& \quad = \left[  \frac{1}{N} \int_0^N\frac{1}{L^2}\int_0^L\left(\int_0^L \la\phi_{u+l_1},\phi_{u+l_2}\ra \diff l_1\right) \diff l_2 \diff u\right]^{1/2}\\
& \quad = \left[ \frac{1}{N} \int_0^N\frac{1}{L^2}\int_0^L \int_0^L \la\phi_{u},\phi_{u+l_2-l_1}\ra \diff l_1 \diff l_2 \diff u\right]^{1/2} \\
& \quad \leq  \left[ \frac{1}{N} \int_0^N\frac{1}{L}\int_{-L}^L \la\phi_{u},\phi_{u+l}\ra \diff l \diff u\right]^{1/2} =  \left[\frac{2}{N} \int_0^N\frac{1}{L}\int_0^L |\la\phi_{u},\phi_{u+l}\ra| \diff l \diff u\right]^{1/2},
\end{split}
\end{equation*}
where we use invariance and the fact that $-L \leq l_1 - l_2 \leq L$. This finishes the proof.
\end{proof}

The following observations will be important in what follows. 
\begin{remark}\label{rem:sobnom}  
There exist a constant $D>0$ such that for every $f\in W^6(M)$ and every $r \geq 1$, we have
$$
\|f\cdot f\circ h_r\|_{6}\leq D\|f\|_6^2 r^{6}.
$$
This follows from the fact that functions in $W^6(M)$ have the algebra property, i.e. $\|f\cdot g\|_6\leq D'\|f\|_6\|g\|_6$ and the fact that $\|f\circ h_r\|_{6}\leq D''\|f\|_6r^6$.
\end{remark}

\begin{remark}\label{rem:notation}
Given $f \in L^2(M)$, let $f^\perp:= f- (\int_M f \diff \mu^\tau) \in L^2_0(M)$. Then $\|f^\perp\|_6 \leq 2 \|f\|_6$.
\end{remark}

Using the van der Corput inequality in Lemma \ref{lem:vdc}, we can prove the important estimate below. Proposition \ref{lem:3mixcase} will be generalized in the Appendix \ref{sec:appendix}, see Proposition \ref{lem:l2}.

\begin{proposition}\label{lem:3mixcase}
For every $0<\alpha \leq 3/2$ there exists $\gamma>0$ such that for any $f_1,f_2 \in W^6(M) \cap L^2_0(M)$ with $\|f_1\|_6, \|f_2\|_6 \leq 1$, for any $n \neq m$ and for all $0<K<1$ satisfying $K >|n-m|^{- \alpha}$, we have
\begin{equation}%\label{eq:l2b}
\left\|\frac{1}{n-m}\int_m^n  (f_1 \circ h^\tau_{K u})\cdot( f_2 \circ h^\tau_u)  \diff u\right\|_{2}\leq C_4 \left(|n-m| (1-K) \right)^{-\gamma},
\end{equation}
for some constant $C_4>0$.
\end{proposition}
\begin{proof}
Notice that we always have 
$$
\left\|\frac{1}{n-m}\int_m^n  (f_1 \circ h^\tau_{K u})\cdot( f_2 \circ h^\tau_u)  \diff u\right\|_{2} \leq \|f_1\|_{\infty} \|f_2\|_{\infty} \leq \|f_1\|_{6} \|f_2\|_{6}\leq 1,
$$
so we can assume that $|n-m|(1-K) \geq 1$. Up to replacing $m$ with $n$, we can also assume that $n>m$.

We will use van der Corput inequality (see Lemma \ref{lem:vdc}) with $N=n-m$ and 
$$
\phi_u(x) := (f_1 \circ h^\tau_{K (u+m)})(x)\cdot( f_2 \circ h^\tau_{u+m})(x) \in L^2(M,\mu).
$$
Since $\|f_i\|_{6}\leq 1$, for every $u\in \R$, we have $\|\phi_u\|_{2} \leq \|f_1\|_{\infty} \|f_2\|_{\infty} \leq 1$. So, by Lemma \ref{lem:vdc}, we have to bound 
 $$
\frac{1}{L}\int_{0}^L|\la \phi_u,\phi_{u+l}\ra|\diff l,
 $$
and optimize for $0 < L \leq N = n-m$. 

With the notation introduced in Remark \ref{rem:notation}, we can estimate
\begin{equation*}
\begin{split}
|\la \phi_u,\phi_{u+l}\ra| = & \left|\int_{M} ( f_1\cdot f_1\circ h^\tau_{K l} ) \circ h^\tau_{K (u+m)}  \cdot (f_2\cdot f_2\circ h^\tau_{l})\circ h^\tau_{u+m}  \diff \mu^\tau \right|\\
\leq & \left|\int_M f_1\cdot (f_1\circ h_{K l}) \diff \mu^\tau \right| \cdot \left|\int_M f_2\cdot (f_2\circ h_{l}) \diff \mu^\tau \right| \\
&+ \left|\int_{M} ( f_1\cdot f_1\circ h^\tau_{K l} )^\perp \circ h^\tau_{K (u+m)}  \cdot (f_2\cdot f_2\circ h^\tau_{l})^\perp \circ h^\tau_{u+m}  \diff \mu^\tau \right|,
\end{split}
\end{equation*}
where we used the fact that, by definition, $( f_1\cdot f_1\circ h^\tau_{K l} )^\perp$ and $(f_2\cdot f_2\circ h^\tau_{l})^\perp$ have zero average.
Applying the mixing estimates of Lemma \ref{lem:fu}, since $\|f_i\|_6 \leq 1$, we get
\begin{equation*}
\begin{split}
|\la \phi_u,\phi_{u+l}\ra| & \leq \frac{C_0^2}{(Kl)^{\beta} l^{\beta}} + \left|\int_{M} ( f_1\cdot f_1\circ h^\tau_{K l} )^\perp \cdot (f_2\cdot f_2\circ h^\tau_{l})^\perp \circ h^\tau_{(u+m)(1-K)} \diff \mu^\tau \right| \\
&\leq \frac{C_0^2}{ K^{\beta} l^{2\beta}} + C_0 \| (f_1\cdot f_1\circ h^\tau_{K l})^\perp \|_6 \| (f_2\cdot f_2\circ h^\tau_{l})^\perp \|_6 \frac{1}{(u+m)^\beta (1-K)^\beta}.
\end{split}
\end{equation*}
From Remarks \ref{rem:sobnom}, \ref{rem:notation}, and $\|f_i\|_6 \leq 1$, it follows that 
$$
\| (f_1\cdot f_1\circ h^\tau_{K l})^\perp \|_6 \leq 2\| f_1\cdot f_1\circ h^\tau_{K l} \|_6 \leq 2 D(Kl)^6,
$$
and similarly for $ \| (f_2\cdot f_2\circ h^\tau_{l})^\perp \|_6$, so that 
$$
|\la \phi_u,\phi_{u+l}\ra| \leq C' \left(\frac{1}{(Kl)^{\beta} l^{\beta}} + \frac{K^6l^{12}}{(u+m)^\beta (1-K)^\beta} \right)
$$
where we can take $C' = \max \{C_0^2, 4D^2C_0\}$.

By Lemma \ref{lem:vdc}, recalling that $N=n-m$, we obtain
\begin{equation}\label{eq:vdc3mixing}
\begin{split}
& \left\|\frac{1}{n-m}\int_m^n  (f_1 \circ h^\tau_{K u})\cdot( f_2 \circ h^\tau_u)  \diff u\right\|_{2} \\
& \qquad \leq 2C' \left[\frac{1}{K^{\beta} L^{2\beta}} + \frac{K^6L^{12}}{ (1-K)^\beta N} \int_0^N \frac{\diff u}{(u+m)^\beta} \right]^{1/2} + O\left( \frac{L}{N} \right) \\
& \qquad \leq 2C' \left[\frac{1}{K^{\beta} L^{2\beta}} + \frac{K^6L^{12}}{ (1-K)^\beta (1-\beta) N^\beta} \right]^{1/2} + O\left( \frac{L}{N} \right).
\end{split}
\end{equation}
By assumption, there exists $0<\alpha \leq 3/2$ such that $K > N^{-\alpha}$. We fix
$$
L = \frac{\left( (1-K)N\right)^{\beta/24}}{K^{1/2}} \leq N^{\beta/24 +\alpha/2} \leq N^{1/24+3/4} < N,
$$
so that, moreover,
$$
\frac{L}{N} \leq \frac{1}{N^{5/24}} \leq \frac{1}{\left( (1-K)N\right)^{5/24}}.
$$
Thus, the term $O(L/N)$ in the right hand-side of \eqref{eq:vdc3mixing} satisfies an estimate of the desired form. It remains to bound the two summands in the square brackets in \eqref{eq:vdc3mixing}. 
By the choice of $L$, we get
$$
\frac{1}{K^\beta L^{2\beta}} \leq \frac{1}{\left( (1-K)N\right)^{\beta^2/12}},
$$
and
$$
\frac{K^6L^{12}}{ (1-K)^\beta (1-\beta) N^\beta} \leq \frac{\left( (1-K)N\right)^{\beta/2} }{ (1-K)^\beta (1-\beta) N^\beta} \leq \frac{2}{(1-\beta)\left( (1-K)N\right)^{\beta/2}},
$$
where we use the fact that $(1-K)N \geq 1$. This concludes the proof.
\end{proof}

\section{Polynomial 3-mixing}\label{sec:3mixproof}

This section is devoted to the proof of Theorem \ref{thm:main3}.
The strategy of the proof is similar to the proof of Theorem 1.1 in \cite{BEG}; however, since we are in the non algebraic setting, our reasoning is local and we use estimates on stretching of geodesic arcs of the time changed flow. We also use some ideas from Marcus' proof in \cite{Mar}.

The first step is to exploit the shearing property of the horocycle flow and its time changes: transverse segments in the geodesic direction get sheared by $h^{\tau}_t$. We will fix $\sigma = t_2^{-(1-\beta/3)}>0$ the length of such segments. 
The proof will be divided in two cases (\textbf{Case A} and \textbf{Case B} below), depending on the relative size of the gaps $t_1$ and $t_2-t_1$.
Roughly speaking, if $t_1$ is \lq\lq much smaller\rq\rq\ than $t_2-t_1$ (\textbf{Case A}), our choice of $\sigma$ will ensure that length of the sheared arc $h^{\tau}_{t_1} \circ g_s$ for time $t_1$ is sufficiently small, so that the correlations can be estimated, up to a small error, by the integral of $f_2$ along the arc $h^{\tau}_{t_2} \circ g_s$.
If $t_1$ and $t_2-t_1$ are \lq\lq of the same order\rq\rq\ (\textbf{Case B}), we will reduce the problem of estimating the multiple correlations to the setting of Proposition \ref{lem:3mixcase}, namely to a \lq\lq multiple ergodic integral\rq\rq.

\begin{proof}[Proof of Theorem \ref{thm:main3}]
Let $0=t_0<t_1<t_2$ be fixed. Up to considering the inverse flow $(h^\tau)^{-1}_t = h^\tau_{-t}$, composing with $h^\tau_{-t_2}$ and relabeling $t_{2-i}'=t_2 - t_i$, we can assume that $t_1\leq t_2-t_1$, so that, in particular $t_2-t_1 \geq t_2 / 2$. We will also assume that $t_1 \geq 1$

Let $f_0,f_1,f_2 \in W^6(M)$; define $C_f =  \|f_0\|_6 \|f_1\|_6 \|f_2\|_6$ and $C_{f,\tau} = \|\tau\|_6 C_f \geq C_f$.
Recalling the notation introduced in Remark \ref{rem:notation}, we have
\begin{equation*}
\begin{split}
&\left\lvert \int_M f_0 \cdot f_1 \circ h^\tau_{t_1} \cdot f_2 \circ h^\tau_{t_2} \diff \mu^\tau  - \left(\int_M f_0 \diff \mu^\tau \right)\left(\int_M f_1 \diff \mu^\tau \right)\left(\int_M f_2 \diff \mu^\tau \right)\right\rvert \\
& \quad \leq \sum_{i=0,1,2} \left\lvert \int_M f_i \diff \mu^\tau  \right\rvert \cdot \left\lvert \int_M \prod_{\{j,k\} = \{0,1,2\} \setminus \{i\}}f_j^\perp \cdot f_k^\perp \circ h^\tau_{t_k- t_j} \diff \mu^\tau \right\rvert +\\
& \quad \quad + \left\lvert \int_M f_0^\perp \cdot f_1^\perp \circ h^\tau_{t_1} \cdot f_2^\perp \circ h^\tau_{t_2} \diff \mu^\tau \right\rvert.
\end{split}
\end{equation*}
By the mixing estimates of Lemma \ref{lem:fu}, we can bound each term in the sum in the right hand-side above by
$$
\left\lvert \int_M f_i \diff \mu^\tau  \right\rvert \cdot \left\lvert \int_M \prod_{\{j,k\} = \{0,1,2\} \setminus \{i\}}f_j^\perp \cdot f_k^\perp \circ h^\tau_{t_k- t_j} \diff \mu^\tau \right\rvert \leq C_0 \|f_i\|_\infty \|f_j^\perp\|_6 \|f_k^\perp\|_6 |t_k-t_j|^{-\beta}, 
$$
thus, by Remark \ref{rem:notation}, we have 
\begin{equation*}
\begin{split}
&\left\lvert \int_M f_0 \cdot f_1 \circ h^\tau_{t_1} \cdot f_2 \circ h^\tau_{t_2} \diff \mu^\tau  - \left(\int_M f_0 \diff \mu^\tau \right)\left(\int_M f_1 \diff \mu^\tau \right)\left(\int_M f_2 \diff \mu^\tau \right)\right\rvert \\
& \quad \leq 12C_0 C_f \frac{1}{t_1^{\beta}} + \left\lvert \int_M f_0^\perp \cdot f_1^\perp \circ h^\tau_{t_1} \cdot f_2^\perp \circ h^\tau_{t_2} \diff \mu^\tau \right\rvert,
\end{split}
\end{equation*}
since $t_1 = \min \{ t_1, t_2-t_1\}$.
Therefore, it remains to bound the correlations for functions of zero average; we will simply denote $f_i$ instead of $f_i^\perp$, and we will assume that $f_i \in W^6(M) \cap L^2_0(M)$.

Let us define 
$$
0<\sigma:=\frac{1}{t_2^{1-\beta/3}}<1.
$$
We recall that the invariant measure $\mu^\tau$ is equivalent to the Haar measure $\mu$, with density $\tau$.
By invariance of $\mu$ under the geodesic flow, we have
\begin{equation}\label{eq:3mixshear}
\begin{split}
\left\lvert \int_M f_0 \cdot f_1 \circ h^\tau_{t_1} \cdot f_2 \circ h^\tau_{t_2} \diff \mu^\tau \right\rvert &= 
\left\lvert \int_M (\tau f_0) \cdot f_1 \circ h^\tau_{t_1} \cdot f_2 \circ h^\tau_{t_2} \diff \mu \right\rvert \\
&= \left\lvert \frac{1}{\sigma} \int_0^\sigma \int_M (\tau f_0) \circ g_s \cdot f_1 \circ h^\tau_{t_1}\circ g_s \cdot f_2 \circ h^\tau_{t_2}\circ g_s \diff \mu \diff s \right\rvert.
\end{split}
\end{equation}
For all $s \in [0,\sigma]$, we have 
$$
\|(\tau f_0) \circ g_s - \tau f_0\|_{\infty} \leq \|\tau f_0\|_6 s \leq \|\tau\|_6 \| f_0\|_6 \sigma,
$$
hence
\begin{equation}\label{eq:forcaseb}
\begin{split}
&\left\lvert \frac{1}{\sigma} \int_0^\sigma \int_M (\tau f_0) \circ g_s \cdot f_1 \circ h^\tau_{t_1}\circ g_s \cdot f_2 \circ h^\tau_{t_2}\circ g_s \diff \mu \diff s \right\rvert \\
& \quad = \left\lvert \int_M  \frac{1}{\sigma} \int_0^\sigma (\tau f_0) \circ g_s \cdot f_1 \circ h^\tau_{t_1}\circ g_s \cdot f_2 \circ h^\tau_{t_2}\circ g_s \diff s \diff \mu \right\rvert \\
& \quad \leq \left\lvert \int_M \tau f_0 \left( \frac{1}{\sigma} \int_0^\sigma f_1 \circ h^\tau_{t_1}\circ g_s \cdot f_2 \circ h^\tau_{t_2}\circ g_s \diff s \right) \diff \mu \right\rvert + \|\tau\|_6 \| f_0\|_6 \|f_1\|_\infty \|f_2\|_\infty \sigma \\
& \quad \leq \left\lvert \int_M f_0 \left( \frac{1}{\sigma} \int_0^\sigma f_1 \circ h^\tau_{t_1}\circ g_s \cdot f_2 \circ h^\tau_{t_2}\circ g_s \diff s \right) \diff \mu^\tau \right\rvert + C_{f,\tau} \sigma.
\end{split}
\end{equation}
We now estimate the first term in the right hand-side above in two different ways, depending on $t_1$.

\textbf{Case A.} 
Let us assume that $t_1\leq t_2^{1-\beta/2}$.
By Lemma \ref{lem:erren} and the triangle inequality,
\begin{equation*}
\begin{split}
\| f_1 \circ h^\tau_{t_1}\circ g_s - f_1 \circ h^\tau_{t_1}\|_\infty \leq &\| f_1 \circ g_s \circ h^\tau_{e^s t_1 + A(\cdot, s, t_1)} - f_1 \circ h^\tau_{e^s t_1 + A(\cdot, s, t_1)}\|_\infty \\
& +\| f_1 \circ h^\tau_{e^s t_1 + A(\cdot, s, t_1)} - f_1 \circ h^\tau_{t_1}\|_\infty, 
\end{split}
\end{equation*}
with $\|A(\cdot, s, t_1)\|_\infty \leq C_3 s t_1^{1-\beta}$. 
Therefore,
\begin{equation*}
\begin{split}
\| f_1 \circ h^\tau_{t_1}\circ g_s - f_1 \circ h^\tau_{t_1}\|_\infty \leq \|f_1\|_6 (s + (e^s-1)t_1 +  C_3 s t_1^{1-\beta}) \leq (3+C_3) \|f_1\|_6 \sigma t_1.
\end{split}
\end{equation*}
We obtain that
\begin{equation}\label{eq:endcasea}
\begin{split}
& \left\lvert \int_M f_0 \left( \frac{1}{\sigma} \int_0^\sigma f_1 \circ h^\tau_{t_1}\circ g_s \cdot f_2 \circ h^\tau_{t_2}\circ g_s \diff s \right) \diff \mu^\tau \right\rvert \\
& \quad \leq \left\lvert \int_M f_0 \cdot f_1 \circ h^\tau_{t_1} \left( \frac{1}{\sigma} \int_0^\sigma f_2 \circ h^\tau_{t_2}\circ g_s \diff s \right) \diff \mu^\tau \right\rvert + (3+C_3)  \|f_0\|_\infty \|f_1\|_6 \|f_2\|_\infty  \frac{1}{t_2^{\beta/6}}.
\end{split}
\end{equation}
Lemma \ref{lem:fu18}  gives us a uniform bound for the term in brackets in \eqref{eq:endcasea}. 
Combining \eqref{eq:3mixshear}, \eqref{eq:forcaseb}, \eqref{eq:endcasea}, and using Lemma \ref{lem:fu18}, we conclude
\begin{equation*}
\begin{split}
\left\lvert \int_M f_0 \cdot f_1 \circ h^\tau_{t_1} \cdot f_2 \circ h^\tau_{t_2} \diff \mu^\tau \right\rvert &\leq C_1 \|f_0\|_\infty \|f_1\|_\infty \|f_2\|_6 \frac{1}{(\sigma t_2)^\beta} + (3+C_3) C_{f} \frac{1}{t_2^{\beta/6}} + C_{f,\tau} \frac{1}{t_2^{1- \beta/3}} \\
&\leq C_5 C_{f,\tau} \frac{1}{t_2^{\gamma}},
\end{split}
\end{equation*}
for some constant $C_5>0$ and where $\gamma = \min\{ \beta^2/3, \beta/6 \}$. 
This concludes the proof for Case A.

\textbf{Case B.}
Let us now assume that $t_1 >  t_2^{1-\beta/2}$.
From \eqref{eq:forcaseb} and Cauchy-Schwartz inequality, we get
\begin{equation}\label{eq:3mixcaseb}
\left\lvert \int_M f_0 \cdot f_1 \circ h^\tau_{t_1} \cdot f_2 \circ h^\tau_{t_2} \diff \mu^\tau \right\rvert \leq \|f_0\|_2 \left\| \frac{1}{\sigma} \int_0^\sigma f_1 \circ h^\tau_{t_1}\circ g_s \cdot f_2 \circ h^\tau_{t_2}\circ g_s \diff s \right\|_2 + C_{f,\tau} \sigma.
\end{equation}
For any point $x \in M$ and any $0 \leq s \leq \sigma<1$, 
\begin{equation*}
\begin{split}
& |f_1 \circ h^\tau_{t_1}\circ g_s (x) \cdot f_2 \circ h^\tau_{t_2}\circ g_s (x)| = |f_1\circ g_s \circ h^\tau_{e^s t_1 + A(x,s,t_1)} (x) \cdot f_2\circ g_s \circ h^\tau_{e^s t_2 +A(x,s,t_2)}(x)| \\
&\quad \leq |f_1 \circ h^\tau_{e^s t_1 + A(x,s,t_1)} (x) \cdot f_2 \circ h^\tau_{e^s t_2 +A(x,s,t_2)}(x)| + \|f_1\|_6\|f_2\|_6 (2s+s^2).
\end{split}
\end{equation*}
By Lemma \ref{lem:erren}, 
$$
\max \{|A(x,s,t_1)|, |A(x,s,t_2)| \} \leq st_2^{1-\beta} \leq \sigma t_2^{1-\beta} = t_2^{-2\beta/3}.
$$
This implies that 
\begin{equation*}
\begin{split}
& |f_1 \circ h^\tau_{e^s t_1 + A(x,s,t_1)} (x) \cdot f_2 \circ h^\tau_{e^s t_2 +A(x,s,t_2)}(x)|   \leq  |f_1 \circ h^\tau_{e^s t_1} (x) \cdot f_2 \circ h^\tau_{e^s t_2}(x)| + \\
& \quad \quad + \|f_1\|_6\|f_2\|_6 (|A(x,s,t_1)| + |A(x,s,t_2)| + |A(x,s,t_1)| \cdot |A(x,s,t_2)| ) \\
& \quad \leq |f_1 \circ h^\tau_{e^s t_1} (x) \cdot f_2 \circ h^\tau_{e^s t_2}(x)| + 3 \|f_1\|_6\|f_2\|_6 t_2^{-2\beta/3},
\end{split}
\end{equation*}
therefore
$$
|f_1 \circ h^\tau_{t_1}\circ g_s (x) \cdot f_2 \circ h^\tau_{t_2}\circ g_s (x)| \leq |f_1 \circ h^\tau_{e^s t_1} (x) \cdot f_2 \circ h^\tau_{e^s t_2}(x)| + 6\|f_1\|_6\|f_2\|_6 |t_2|^{-2\beta/3}.
$$
Together with \eqref{eq:3mixcaseb}, we get
\begin{equation}\label{eq:3mixcasebend}
\left\lvert \int_M f_0 \cdot f_1 \circ h^\tau_{t_1} \cdot f_2 \circ h^\tau_{t_2} \diff \mu^\tau \right\rvert \leq \|f_0\|_2 \left\| \frac{1}{\sigma} \int_0^\sigma f_1 \circ h^\tau_{e^s t_1} \cdot f_2 \circ h^\tau_{e^s t_2} \diff s \right\|_2 + 7C_{f,\tau} |t_2|^{-2\beta/3}.
\end{equation}

We now estimate the first term in the right hand-side above using Proposition \ref{lem:3mixcase}.
Define $0< K= t_1/t_2 <1$. For all $x \in M$, changing variable $u = e^st_2$ and by the second mean-value theorem for integrals, there exists $z\in [t_2,e^\sigma t_2]$ such that
\begin{equation*}
\begin{split}
& \left\lvert \int_0^\sigma f_1 \circ h^\tau_{e^s t_1}(x) \cdot f_2 \circ h^\tau_{e^s t_2}(x) \diff s \right\rvert = \left\lvert \int_{t_2}^{e^\sigma t_2} f_1 \circ h^\tau_{Ku}(x) \cdot f_2 \circ h^\tau_{u}(x) \frac{\diff u}{u} \right\rvert  \\
& \quad  = \left\lvert \frac{1}{t_2} \int_{t_2}^{z} f_1 \circ h^\tau_{Ku}(x) \cdot f_2 \circ h^\tau_{u}(x) \diff u+\frac{1}{e^\sigma t_2} \int_{z}^{e^\sigma t_2} f_1 \circ h^\tau_{Ku}(x) \cdot f_2 \circ h^\tau_{u}(x) \diff u \right\rvert \\
& \quad \leq \left\lvert \frac{1}{t_2} \int_{t_2}^{e^\sigma t_2} f_1 \circ h^\tau_{Ku}(x) \cdot f_2 \circ h^\tau_{u}(x) \diff u\right\rvert+ \left\lvert\frac{1-e^\sigma}{e^\sigma t_2} \int_{z}^{e^\sigma t_2} f_1 \circ h^\tau_{Ku}(x) \cdot f_2 \circ h^\tau_{u}(x) \diff u \right\rvert,
\end{split}
\end{equation*}
and the second term in the right hand-side above is $\leq \|f_1\|_\infty \|f_2\|_\infty |e^\sigma-1|^2 \leq 4 \|f_1\|_6 \|f_2\|_6\sigma^2$.
This yields
\begin{equation*}
\begin{split}
\left\| \frac{1}{\sigma} \int_0^\sigma f_1 \circ h^\tau_{e^s t_1} \cdot f_2 \circ h^\tau_{e^s t_2} \diff s \right\|_2 \leq & \frac{e^\sigma - 1}{\sigma} \left\| \frac{1}{(e^\sigma-1)t_2} \int_{t_2}^{e^\sigma t_2} f_1 \circ h^\tau_{Ku} \cdot f_2 \circ h^\tau_{u} \diff u \right\|_2 \\
&+ 4 \|f_1\|_6 \|f_2\|_6\sigma.
\end{split}
\end{equation*}
Since we are in the case $t_1 > t_2^{1-\beta/2}$, we have that 
$$
\left\lvert \frac{1}{(e^\sigma-1)t_2} \right\rvert^{3/2} < \frac{1}{(\sigma t_2)^{3/2}}  = \frac{1}{t_2^{\beta/2}} < \frac{t_1}{t_2} = K.
$$
Hence, the assumptions of Proposition \ref{lem:3mixcase} are satisfied. We then get
\begin{equation*}
\begin{split}
\|f_0\|_2 \left\| \frac{1}{\sigma} \int_0^\sigma f_1 \circ h^\tau_{e^s t_1} \cdot f_2 \circ h^\tau_{e^s t_2} \diff s \right\|_2 &\leq CC_f \left( \left( \frac{e^\sigma - 1}{\sigma}  \right)  \frac{1}{|(e^\sigma-1)t_2(1-K)|^\gamma} + \frac{1}{t_2^{1-\beta/3}} \right) \\
& \leq C' C_f \left( \frac{1}{(\sigma(t_2-t_1))^\gamma} + \frac{1}{t_2^{1-\beta/3}} \right). 
\end{split}
\end{equation*}
Since $t_2-t_1 \geq t_2/2$, it follows $(\sigma(t_2-t_1))^{-\gamma} \leq 2^\gamma t_2^{-\beta\gamma/3}$.
This estimate and \eqref{eq:3mixcasebend} conclude the proof.
\end{proof}

%%%%%%%%%%%%%%%%%%%%%%%%%%%%%

\section{Appendix: higher order mixing for discrete-series reparametrizations}\label{sec:appendix}

In this section, we show how to refine the argument above to obtain polynomial mixing of all orders for time-changes $(h_t^\tau)_{t\in\R}$ supported on the discrete series, namely for $\tau\in W^6(M)\cap \mathcal{H}_d$.

%The strategy of the proof is similar to the proof of Theorem 1.1 in \cite{BEG}; however, since we are in the non algebraic setting, our reasoning is local and we use estimates on stretching of geodesic arcs of the time changed flow. We also use some ideas from Marcus' proof in \cite{Mar}. Similarly to \cite{BEG} and \cite{Mar}, the proof goes by induction, i.e.~we assume that we have quantitative $(k-1)$-mixing estimates for smooth functions and, using the dynamics of the time changed flow (or Lemma \ref{lem:erren}) and van der Corput inequality, we obtain quantitative $k$-mixing estimates. We will deal with the time changed flow $(h_t^\tau)_{t\in\R}$ for $\tau\in W^6(M)\cap \mathcal{H}_d$. Without loss of generality, for the rest of the paper we make a standing assumption that the $L^2$ norms of all functions that appear are bounded above by $1$. 

\subsection{Preliminaries}

We start with the following definition.
\begin{definition}[Quantitative $k$-mixing]\label{def:qm} 
Let $(\varphi_t)_{t\in \R}$ be a measure preserving flow acting on $(X,\cB,\mu)$, let $L^2_0(X,\mu) = \{ f \in L^2(X,\mu) : \int_Xf \diff \mu =0 \}$, and let $\mathcal{F}\subset L^2_0(X,\mu)$ be a subspace equipped with a norm $\|\cdot\|_\mathcal{F}$. We say that $(\varphi_t)_{t\in \R}$ has the $Q(k,\mathcal{F})$- property for $k\in \N$ if there exist $\gamma_k>1, \beta_k>0$ such that for every $f_1,\ldots,f_k\in \mathcal{F}$ and every $(t_i)_{i=1}^k$ with $0=t_1\leq t_2\leq\ldots\leq t_k$, or with $t_k \leq \dots \leq t_2 \leq t_1=0$,  we have 
$$
\left|\int_X \prod_{i=1}^k f_i\circ \varphi_{t_i} \diff \mu\right|\leq C\left(\max \Big(1,\prod_{i=1}^k\|f_i\|_{\mathcal{F}} \Big)\right)^{\gamma_k}
\left(\frac{1}{\min_{0< i\leq k} (t_{i+1}-t_i)}\right)^{\beta_k},
$$
for some constant $C=C(k)>0$.
\end{definition}
We will use the above definition for $(h^\tau_t)$ on $(M,\mu)$ and $\mathcal{F}=W^6(M)\cap L^2_0(M,\mu)$ with $\|\cdot\|_\mathcal{F}=\|\cdot\|_6$. To shorten the notation we denote $Q(k)=Q(k,W^6(M)\cap L^2_0(M,\mu))$. Notice that by Lemma \ref{lem:fu}, it follows that $(h^\tau_t)$ has the $Q(2)$-property. The theorem below is a quantitative version of Proposition 1 in \cite{Mar}.

\begin{theorem}\label{main:th2}
Let $\tau\in W^6(M)\cap \mathcal{H}_d$. For every $k\geq 2$ if $(h^\tau_t)_{t\in\R}$ has the property $Q(\ell)$ for every $2\leq \ell\leq k$ then it has the property $Q(k+1)$. Moreover there exists an explicit lower bound on $\beta_k$ in terms of $\tau$,  $M$ and $k$ for every $k\geq 2$.
\end{theorem}

Theorem \ref{thm:main} follows easily from Theorem \ref{main:th2} above.

\begin{proof}[Proof of Theorem \ref{thm:main}]
By Lemma \ref{lem:fu}, the time change $(h^\tau_t)$ has the $Q(2)$-property; hence, by Theorem \ref{main:th2}, it has property $Q(k)$ for all $k\geq 2$. Let $ f_0,\ldots,f_{k-1}\in W^6(M)$. For each $0\leq i \leq k-1$, denote $\bar{f}_i = \int_M f \diff \mu \in \R$ and $f_i^{\perp} = f_i - \bar{f}_i \in W^6(M) \cap L^2_0(M)$. We have
 \begin{equation}\label{asd2}
\begin{split}
& \left|\int_M \prod_{i=0}^{k-1} f_i\circ h^\tau_{t_i} \diff \mu-\prod_{i=0}^{k-1}\int_M f_i \diff \mu\right| = \left|\int_M \prod_{i=0}^{k-1} (f^{\perp}_i + \bar{f}_i )\circ h^\tau_{t_i} \diff \mu-\prod_{i=0}^{k-1}\bar{f}_i \right| \\
& \leq \sum_{J \subseteq \{ 0,\dots,k-1 \},\ |J| \geq 2} \left(\prod_{i \in J^c} |\bar{f}_i | \right) \left|\int_M \prod_{j \in J} f^{\perp}_j\circ h^\tau_{t_j} \diff \mu \right| \\
& \leq C' \sum_{J \subseteq \{ 0,\dots,k-1 \},\ |J| \geq 2} \left(\frac{1}{\min_{j,j' \in J} |t_{j'}-t_j|}\right)^{\beta_{|J|}}  \leq C\left(\frac{1}{\min_{0\leq i<j\leq k-1} |t_i-t_j|}\right)^{\gamma},
\end{split}
\end{equation}
where $\gamma = \min\{ \beta_j : 2\leq j\leq k-1\}$. This concludes the proof.
\end{proof}

\subsection{A van der Corput estimate}

We need to generalize the statement of Proposition \ref{lem:3mixcase} for all $k \geq 2$. The proof is analogous, with additional technical difficulties.

\begin{proposition}\label{lem:l2}
Let $k\in \N$ and assume that $(h^\tau_t)$ satisfies the $Q(\ell)$-property for $2\leq \ell \leq k$. 
There exists $\eta_k>0$ such that  for every $0<\epsilon<1 + 1/k$ there exists $\delta=\delta(\epsilon,k)>0$  such that for any $f_i\in W^6(M)\cap L^2_0(M)$, for $1\leq i\leq k$, for every $n \neq m$, and for every $|n-m|^{-\epsilon}<K_1<\ldots <K_k=1$, we have 
\begin{multline}\label{eq:l2b}
\left\|\frac{1}{n-m}\int_m^n  \prod_{i=1}^kf_i\circ h^\tau_{K_i u} \diff u\right\|_{2}\leq \\C\left(\max \Big( 1,\prod_{i=1}^k\|f_i\|_{6} \Big)\right)^{\eta_k}\left(|n-m|\min_{1\leq i\leq k}(K_{i+1}-K_i)\right)^{-\delta},
\end{multline}
for some constant $C=C_k>0$.
\end{proposition}
 \begin{proof}
As in the proof of Proposition \ref{lem:3mixcase}, we will assume that $n > m$ and $|n-m| \min_i (K_{i+1}-K_i) \geq 1$. 
We will use van der Corput inequality (see Lemma \ref{lem:vdc}). Let $\phi_u(\cdot)=\prod_{i=1}^kf_i\circ h_{K_i(u+m)}(\cdot)\in L^2(M,\mu)$.

Using the notation of Remark \ref{rem:notation}, we can write
\begin{equation}\label{eq:boun}
\begin{split}
&|\la \phi_u,\phi_{u+l}\ra|  = \left|\int_{M}\left(\prod_{i=1}^kf_i\circ h_{K_i(u+m)}(x)\right)\left(\prod_{i=1}^kf_i\circ h_{K_i(u+m+l)}(x)\right) \diff \mu\right|\\
& \quad = \left|\int_M\prod_{i=1}^k(f_i\cdot f_i\circ h_{K_il})(h_{K_i(u+m)}x) \diff \mu\right|\leq \prod_{i=1}^k\left|\int_Mf_i\cdot (f_i\circ h_{K_il}) \diff \mu\right|+ \\
& \quad + \sum_{J \subsetneq \{ 1,\dots,k \}} \left( \prod_{i \in J}\left|\int_Mf_i\cdot (f_i\circ h_{K_il}) \diff \mu\right| \right) \left( \left|\int_M\prod_{j\in J^c} (f_j\cdot f_j\circ h_{K_jl})^{\perp} (h_{K_j(u+m)}x) \diff \mu\right| \right),
\end{split}
\end{equation}
where in the last term we allow $J=\emptyset$ in which case $\prod_{i \in J}\left|\int_Mf_i\cdot (f_i\circ h_{K_il}) \diff \mu\right|:=1$. Notice moreover that if $|J^c|=1$, then the last integral vanishes by measure invariance. Hence we will assume that the sum is taken over $J\subsetneq \{1,\dots, k\}$ with $|J|\leq k-2$.

We will now bound each term on RHS of \eqref{eq:boun}: first trivially, we have
$$
\left|\int_Mf_i\cdot (f_i\circ h_{K_il}) \diff \mu\right|\leq \|f_i\|^2_{6},
$$
therefore the last term in \eqref{eq:boun} is bounded  by
$$
C\left(\prod_{i=1}^k\|f_i\|_6\right)^2 \sum_{\substack{J \subsetneq \{ 1,\dots,k \}\\ |J|\leq k-2}} \left|\int_M\prod_{j\in J^c} (f_j\cdot f_j\circ h_{K_jl})^{\perp} (h_{K_j(u+m)}x) \diff \mu\right| 
$$

%Therefore, for each $J \subseteq \{ 1,\dots,k \}$ with $|J| \geq 1$, since 
%$$
%\left|\int_M\prod_{j\in J^c} (f_j\cdot f_j\circ h_{K_jl})^{\perp} (h_{K_j(u+m)}x) \diff \mu\right| \leq \prod_{j\in J^c} \| (f_j\cdot f_j\circ h_{K_jl})^{\perp}\|_{\infty}\leq 2 \prod_{j\in J^c} \| f_j \|_{\infty}^2
%$$
%we obtain
 %\begin{equation}\label{eq:setJ}
%\begin{split}
%&\prod_{i \in J}\left|\int_Mf_i\cdot (f_i\circ h_{K_il}) d\mu\right| \left( \left|\int_M\prod_{j\in J^c} (f_j\cdot f_j\circ h_{K_jl})^{\perp} (h_{K_j(u+m)}x) \diff \mu\right| \right) \\
%& \quad \leq CC_f\left(\prod_{i \in J}K_i^{-\beta}\right)l^{-|J|\beta},
%\end{split}
%\end{equation}
%where we have denoted $C_f = \prod_{i=1}^k\|f_i\|^2_{6}$.
Let $J$ be as above and 
let $a^J_1<a^J_2<\ldots<a^J_{m_J}$ denote all the elements in $J^c$ (recall that $m_J\geq 2$). Using the $Q(\ell)$-property for $\ell:=|J^c|\leq k$, we can bound the last term in the RHS of \eqref{eq:boun} by
\begin{multline}
\left|\int_M\prod_{j\in J^c} (f_j\cdot f_j\circ h_{K_jl})^{\perp} (h_{K_j(u+m)}x) \diff \mu\right| \leq C\left(\max \Big(1,\prod_{i\in J^c}\| (f_i\cdot f_i\circ h_{K_il})^{\perp}\|_6 \Big)\right)^{\gamma_{|J^c|}}\\
 \times \left(\frac{1}{(u+m)\min_{1\leq i\leq m_j}(K_{a^J_{i+1}}-K_{a^J_i})}\right)^{\beta_{|J^c|}}\leq\\
C\left(\max \Big(1,\prod_{i\in J^c}\| (f_i\cdot f_i\circ h_{K_il})^{\perp}\|_6 \Big)\right)^{\gamma_{|J^c|}}\times \left(\frac{1}{(u+m)\min_{1\leq i\leq k}(K_{i+1}-K_{i})}\right)^{\beta_{|J^c|}}, 
\end{multline}
where in the last inequality we use the fact that the minimum on the LHS is taken over a smaller set than on the RHS ($(a_i^J)$ is a subset of $\{1,\ldots k\}$).

 We now have the following important estimate (see Remark \ref{rem:sobnom}): for every $1\leq i\leq k$,
 $$
 \|(f_i\cdot f_i\circ h_{K_il})^\perp\|_{6}\leq 2\|f_i\cdot f_i\circ h_{K_il}\|_{6}\leq D\|f_i\|^2_{6}(K_il)^{6}.
 $$
Let us define $C_f:=\prod_{i=1}^k\|f_i\|_6^2$.
The inequality above implies that
 $$
 \prod_{i\in J^c}\|(f_i\cdot f_i\circ h_{K_il})^\perp \|_6\leq D^{|J^c|}C_f \left(\prod_{i\in J^c}K_i\right)^6 l^{6|J^c|}.
 $$

The above bounds and \eqref{eq:boun} imply that
\begin{multline*}
\frac{1}{L}\int_{0}^L|\la \phi_u,\phi_{u+l}\ra|\diff l \leq \frac{1}{L}\int_0^L\prod_{i=1}^k\left|\int_Mf_i\cdot (f_i\circ h_{K_il}) \diff \mu\right| \diff l+\\
+C \sum_{\substack{J \subsetneq \{ 1,\dots,k \}\\ |J|\leq k-2}}  \max \left( 1,   (D^{|J^c|}C_f)^{\gamma_{|J^c|}}  \Big( \prod_{i=1}^kK_i^{6\gamma_{|J^c|}} \Big) L^{6|J^c|\gamma_{|J^c|}}\right) \left(\frac{1}{(u+m)\min_i(K_{i+1}-K_i)}\right)^{\beta_{|J^c|}}.
\end{multline*}

%Let 
%$$
%a_{L,J}:= 2 CC_f\left(\prod_{i\in J} K_i^{-\beta}\right)L^{-|J|\beta},
%$$
%and

Let 
$$
a_L:=\frac{1}{L}\int_0^L\prod_{i=1}^k\left|\int_Mf_i\cdot (f_i\circ h_{K_il}) \diff \mu\right| \diff l
$$
and let 
$$
b_{L,J^c}(u):= 2 C \max \left( 1,   (D^{|J^c|}C_f)^{\gamma_{|J^c|}}  \Big( \prod_{i=1}^kK_i^{6\gamma_{|J^c|}} \Big) L^{6|J^c|\gamma_{|J^c|}}\right) \left(\frac{1}{(u+m)\min_i(K_{i+1}-K_i)}\right)^{\beta_{|J^c|}}.
$$
Let $N = n-m$. By Lemma \ref{lem:vdc}, we have
\begin{equation}\label{eq:vdc_step1}
\left\|\frac{1}{n-m}\int_m^n  \prod_{i=1}^kf_i\circ h_{K_i u} \diff u\right\|_{2}\leq \left[2 a_L+\sum_{\substack{J \subsetneq \{ 1,\dots,k \}\\ |J|\leq k-2}} \frac{1}{N}\int_0^N b_{L,J^c}(u) \diff u \right]^{1/2} + O\left(\frac{L}{N}\right).
\end{equation}

Notice that by the $Q(2)$ property (used only for the last term in the product), we have (using also that $K_k=1$)
$$
a_L\leq \prod_{i=1}^{k-1}\|f_i\|_6\frac{1}{L}\int_0^L\left|\int_{M}f_k\cdot f_k\circ h_l d\mu \right|\diff l\leq \prod_{i=1}^{k}\|f_i\|^2_6\frac{1}{L}\int_0^Ll^{-\beta}\diff l\leq C_fL^{-\beta},
$$
where, we recall, $C_f=\prod_{i=1}^k\|f_i\|_6^2$.

We now define $L\in [0,N]$. By assumption, let $0<\varepsilon <(k+1)/k$ be such that $K_1 > N^{-\varepsilon}$. Let us define $\theta=\theta_{k,\varepsilon}$ 
and $L$ by
$$
\theta  = \frac{1}{2} \min \left(k-\epsilon(k-1),\min_{i\leq k}\frac{\beta_i}{12\gamma_i} \right)
$$
and
$$
L:=\left(\min_{\substack{J\subset \{1,\ldots, k\}\\|J|\leq k-2}}[\prod_{i\in J^c} K_i]^{-1}\Big(N\min_i (K_{i+1}-K_i)\Big)^{\theta}\right)^{\frac{1}{|J^c|}},
$$
(where we allow $J=\emptyset$).
This implies that for every $J\subset \{1,\ldots, k\}$ with $|J|\leq k-2$,
\begin{equation}\label{prop:ell}
L^{|J^c|}\prod_{i\in J^c} K_i\leq \left(N\min_i(K_{i+1}-K_i)\right)^{\theta}.
\end{equation}
Moreover, if a set $J_0\subset \{1,\dots k\}$ realizes the minimum in the definition of $L$, then
\begin{equation}\label{eq:kL}
L^k\geq L^{|J_0^c|}\geq L^{|J_0^c|}\prod_{i\in J_0^c} K_i=\left(N\min_i(K_{i+1}-K_i)\right)^\theta.
\end{equation}

From our assumption $N^{-\varepsilon} \leq K_1 \leq \cdots \leq K_k = 1$, we deduce that 
$$
\left(\prod_{i=1}^kK_i\right)^{\frac{1}{k}} \geq N^{-\varepsilon \frac{k-1}{k}};
$$
in particular, by \eqref{prop:ell} for $J=\emptyset$ and using the definition of $\theta$,
 \begin{equation}\label{eq:vdc_bound1}
\begin{split}
0 \leq \frac{L}{N} &\leq \frac{(N\min_i (K_{i+1}-K_i))^{\theta/k}}{\big(\prod_{i=1}^kK_i \big)^{1/k} N}\leq \frac{1}{N^{1-\theta/k-\epsilon (k-1)/k}} \\
& \leq \frac{1}{\big( N \min_i (K_{i+1}-K_i) \big)^{1-\theta/k-\epsilon(k-1)/k}}\leq 1.
\end{split}
\end{equation}

Notice that by the bound on $a_L$ and \eqref{eq:kL} it follows that 
 \begin{equation}\label{eq:vdc_bound2}
a_{L}  \leq C_f\frac{1}{(N\min_i (K_{i+1}-K_i))^{\frac{\theta\beta}{k}}}.
\end{equation}
Fix $J\subsetneq \{1,\dots, k\}$, $|J|\leq k-2$. For the  term $b_{L,J^c}(u)$ in \eqref{eq:vdc_step1}, since 
$$
\int_0^N \frac{\diff u}{(u+m)^{\beta_{|J^c|}}} = \frac{n^{1-\beta_{|J^c|}} - m^{1-\beta_{|J^c|}}}{1-\beta_{|J^c|}} \leq \frac{(n-m)^{1-\beta_{|J^c|}}}{1-\beta_{|J^c|}}= \frac{N^{1-\beta_{|J^c|}}}{1-\beta_{|J^c|}},
$$
it follows, using also \eqref{prop:ell}, that 
 \begin{multline}\label{eq:vdc_bound3}
\frac{2}{N}\int_0^N b_{L,J^c}(u) \diff u \leq \\
\frac{2C}{N} \left( 1+ (D^{|J^c|} C_f)^{\gamma_{|J^c|}}  \Big( \prod_{i\in J^c}K_i^{6\gamma_{|J^c|}} \Big) L^{6|J^c|\gamma_{|J^c|}}\right) \frac{1}{\min_{i}(K_{i+1}-K_i)^{\beta_{|J^c|}}} \int_0^N \frac{\diff u}{(u+m)^{\beta_{|J^c|}}} \\
 \leq C' \frac{1}{\min_i (K_{i+1}-K_i)^{\beta_{|J^c|}}N^{\beta_{|J^c|}}} + C' C_f^{\gamma_{|J^c|}} \frac{1}{(N\min_i (K_{i+1}-K_i))^{\beta_{|J^c|}-6\gamma_{|J^c|} \theta}} \\
\leq C' \max(1,C_f)^{\gamma_{|J^c|}}  \frac{1}{(N\min_i (K_{i+1}-K_i))^{\beta_{|J^c|}-6\gamma_{|J^c|} \theta}}.
\end{multline}

Set 
$$
\eta_k:=\frac{1}{2}\max(1,\max_{i\leq k}\gamma_i)
$$
and
$$
\delta:=\min\Big(1-\frac{\theta}{k}-\epsilon \frac{k-1}{k}, \frac{\theta\beta}{2k} , \min_{i\leq k}\frac{ \beta_i-6\gamma_i \theta}{2} \Big),
$$
notice that $\delta>0$ by the definition of $\theta$.
We get from \eqref{eq:vdc_step1}, using \eqref{eq:vdc_bound1}, \eqref{eq:vdc_bound2}, and \eqref{eq:vdc_bound3} (for each $J\subsetneq \{1,\dots, k\}$ and  summing over $J$), we obtain
$$
\left\|\frac{1}{n-m}\int_m^n  \prod_{i=1}^kf_i\circ h_{K_i u} \diff u\right\|_{2} \leq C'' \max(1,C_f)^{\eta_k} \frac{1}{N^{\delta}\min_i (K_{i+1}-K_i)^{\delta}}. 
$$
\end{proof}

\subsection{Combinatorial argument}

We describe an inductive procedure that will be used in the proof of Theorem \ref{main:th2} (see the outline of the proof below).

Fix $k\in \N$ and fix numbers $(\zeta_i)_{i\leq k}$ in $(0,1)$. For any $(k+1)$-tuple of numbers $0=t_0< t_1<...<t_k$ we perform the following inductive procedure. 
\begin{itemize}[align=left]
\item[\emph{Step 1.}] Let $r_1=t_k^{\frac{\zeta_1}{12k}}$. If $\{t_i\}_{i=0}^k\subset [0,r_1]\cup[t_k-r_1,t_k]$, the procedure stops. If not, let $s_1<k$ be the largest such that $t_{s_1}\notin [0,r_1]\cup[t_k-r_1,t_k]$.
\item[\emph{Step 2.}]  Let $r_2:=r_1^{\frac{\zeta_2}{12k}}$. If $\{t_i\}_{i=0}^k\subset [0,r_2]\cup[t_{s_1}-r_2,t_{s_1}+r_2] \cup[t_k-r_2,t_k]$ the procedure stops. If not, let $s_2<k$ be the largest such that $t_{s_2}\notin [0,r_2] \cup [t_{s_1}-r_2,t_{s_1}+r_2] \cup[t_k-r_2,t_k]$.
\item[\emph{Step $\ell+1$.}] If the procedure does not stop at {\em Step $\ell$}, let $r_{\ell+1}:=r_\ell^{\frac{\zeta_{\ell+1}}{12k}}$ and take ${s_\ell}<k$ to be the largest such that $t_{s_\ell}\notin  [0,r_{\ell}] \cup \bigcup_{m=1}^{\ell-1}[t_{s_m}-r_\ell,t_{s_m}+r_\ell] \cup [t_k-r_\ell,t_k]$.
\end{itemize} 
%Now, inductively, if the procedure does not stop at {\em Step $\ell$}, we take ${s_\ell}<k$ to be the largest such that $t_{s_\ell}\notin  [0,r_{\ell}]\cup [t_k-r_\ell,t_k]\bigcup_{m=1}^{\ell-1}[t_{s_m}-r_\ell,t_{s_m}+r_\ell]$ and let $r_{\ell+1}:=r_\ell^{\frac{\zeta_{\ell+1}}{12k}}$. 
Notice that the procedure will definitely stop no later than {\em Step $k$}. Moreover, notice that if the procedure stops exactly at {\em Step $k$}, then by the definition of $(r_\ell)$, for $\xi_k:=\frac{\prod_{i=1}^k \zeta_i}{(12k)^k}$, we have
\be\label{eq:stopk}
\min_{0\leq i <k}|t_{i+1}-t_i|\geq t_k^{\xi_k}.
\ee
It is crucial that $\xi_k$ depends on $(\zeta_i)$ and $k$ but \emph{not} on the $(t_i)_{i\leq k}$.

\subsection{Proof of Theorem \ref{main:th2}}\label{sec:proof_of_Theorem_2}

The rest of the section is devoted to the proof of Theorem \ref{main:th2}. We first present an outline for the reader's convenience.

\paragraph{Outline of the proof.}
The proof consists of two cases (\textbf{Case A} and \textbf{Case B} below), depending whether the inductive procedure described above stops at \emph{Step $k$} or before. 

If it stops at \emph{Step $\ell$} for $\ell < k$ (\textbf{Case A}), it means that there exist $j \neq i$ such that the corresponding times  $t_j$ and $t_i$ are close, namely $|t_i-t_j| \leq 2 r_{\ell}$. We then write 
$$
f_i \circ h^{\tau}_{t_i} \cdot f_j \circ h^{\tau}_{t_j} = (f_i \cdot f_j \circ h^{\tau}_{t_j-t_i}) \circ h^{\tau}_{t_i},
$$
and, by assumption, the Sobolev norm of the term in brackets is small, namely is of order $O(r_\ell^{6})$. 
We do the same for all the times $t_j$ contained in an interval of the form $[t_{s_i}-r_{\ell}, t_{s_i}+r_{\ell}]$ as described in the inductive procedure, and we consider the corresponding terms in brackets as a single observable. In this way, we reduce the number of observables to $\ell<k$ (with appropriate bounds on their Sobolev norms), and we can apply the inductive hypothesis on quantitative $\ell$-mixing to conclude.

If the procedure does stop exactly at \emph{Step $k$} (\textbf{Case B}), 
we proceed as in the proof of Theorem \ref{thm:main3}, exploiting the shearing properties of geodesic segments of length $\sigma$.
We remark that our assumption on the time-change ensures that the deviations of the shearing property form the unperturbed homogeneous case is logarithmic (see Lemma \ref{lem:erren}), hence the error term is of order $\sigma \log^k |t_k|$. 
In this case, for the assumptions of Proposition \ref{lem:l2} to be satisfied, we will need to choose $\sigma = |t_k|^{-\alpha^2}$ for some small $\alpha>0$. In order to conclude, it will be crucial to exploit \eqref{eq:stopk}, which will ensures that $\sigma = |t_k|^{-\alpha^2} = O(\min_{0\leq i <k}|t_{i+1}-t_i|^{-\widetilde{\alpha}})$, for some $\widetilde{\alpha}>0$.

\begin{proof}[Proof of Theorem \ref{main:th2}] 
Let $k \geq 2$ and assume that $(h^{\tau}_t)$ has the property $Q(\ell)$ for $2\leq \ell \leq k$. Let $\zeta_{i+1}:=\frac{\beta_{i+1}}{\gamma_{i+1}}$.
Fix $f_0,\ldots, f_{k}\in W^6(M) \cap L^2_0(M)$ and let $0=t_0\leq t_1\leq\ldots \leq t_{k}$, or $t_k\leq \ldots \leq t_1 \leq t_{0}=0$ .
By invariance of the measure $\mu^{\tau}$ for $(h^{\tau}_t)$, up to composing with $h^{\tau}_{-t_k}$ and relabeling $t_{k-i}' = t_k-t_i$, we can assume that
\begin{equation}\label{eq:thm2_max}
\min_{0\leq i <k} |t_{i+1}-t_i| = \min_{1\leq i < k} |t_{i+1}-t_i|.
\end{equation}

We now apply the combinatorial procedure to the sequence $(\zeta_i)_{i\leq k-1}$ and $t_0<...<t_k$. Assume the procedure stops at {\em Step $\ell$}.  We consider two cases:

\textbf{Case A.} $\ell<k$. By definition this means that $\{t_i\}_{i=0}^k\subset [0,r_{\ell}]\cup [t_k-r_\ell,t_k]\bigcup_{m=1}^{\ell-1}[t_{s_m}-r_\ell,t_{s_m}+r_\ell]$. For $0\leq i\leq \ell$, let 
$$
\tilde{f}_i:=\prod_{t_j\in [t_{s_i}-r_\ell,t_{s_i}+r_\ell]} f_j\circ h_{t_j-t_{s_i}},
$$
with $s_0=0$ and $s_\ell=k$. Let $w_i:=\#\{j\leq \ell\;:\; t_j\in [t_{s_i}-r_\ell,t_{s_i}+r_\ell]\}$. Then 
\begin{equation}\label{boun:int}
\int_M \prod_{i=0}^kf_i \circ h^{\tau}_{t_i} \diff \mu^{\tau}=\int_M 
\prod_{i=0}^{\ell}\tilde{f}_i\circ h^\tau_{t_{s_i}}\diff \mu^{\tau}.
\ee
Notice that by the same splitting as in \eqref{asd2}, we have
\begin{equation}\label{eq:neqw}
\left| \int_M \prod_{i=0}^{\ell}\tilde{f}_i\circ h^\tau_{t_{s_i}} \diff \mu^{\tau} \right| \leq C\left(\prod_{i=0}^k \|f_i\|_\infty\right)\sum_{i=0}^\ell \left|\int_M\tilde{f}_id\mu^\tau\right|
+ \left| \int_M \prod_{i=0}^{\ell}\tilde{f}_i^{\perp}\circ h^\tau_{t_{s_i}}d\mu^\tau \right|.
\end{equation}
Moreover, by the $Q(w_i)$ property and the definition of $\tilde{f}_i$,
\begin{multline*}
\left|\int_M\tilde{f}_i \diff \mu^\tau\right|<C\left(\prod_{i=0}^k \|f_i\|_6\right)^{w_i} \left(\frac{1}{\min_{t_j\in [t_{s_i}-r_\ell,t_{s_i}+r_\ell]} (t_{j+1}-t_j)}\right)^{\beta_{w_i}}\leq \\
C\left(\prod_{i=0}^k \|f_i\|_6\right)^{w_i}\left(\frac{1}{\min_{i\leq k}(t_{i+1}-t_i)}\right)^{\beta_{w_i}}.
\end{multline*}
Therefore, these terms have the desired behaviour. We will now deal with the last term in \eqref{eq:neqw}.
Since $\ell<k$, we can use the $Q(\ell+1)$ property to bound 
$$
\left| \int_M \prod_{i=0}^{\ell}\tilde{f}_i^{\perp}\circ h^\tau_{t_{s_i}} \diff \mu^\tau \right| \leq C\left(\prod_{i=0}^\ell \|\tilde{f}_i\|_6\right)^{\gamma_{\ell+1}}\times \\
\left(\frac{1}{\min_{0\leq i\leq \ell} t_{s_{i+1}}-t_{s_i}}\right)^{\beta_{\ell+1}}.
$$ Notice that by the definition of $(\tilde{f}_i)$,
$$
\prod_{i=0}^\ell \|\tilde{f}_i\|_6\leq (\prod_{i=0}^k \|f_i\|_6)(r_\ell)^{6w_i}\leq (\prod_{i=0}^k \|f_i\|_6)(r_\ell)^{6k}. 
$$
On the other hand, by the definition of $(s_i)$, we have 
$\min_{0\leq i\leq \ell} t_{s_{i+1}}-t_{s_i}\geq r_{\ell-1}$. Therefore, we can bound \eqref{boun:int} by
$$
\left| \int_M \prod_{i=0}^{\ell}\tilde{f}_i\circ h^\tau_{t_{s_i}}\diff \mu^{\tau} \right|  \leq C(\prod_{i=0}^k \|f_i\|_6)^{6\gamma_{\ell+1}}r_{\ell}^{6k\gamma_{\ell+1}} r_{\ell-1}^{-\beta_{\ell+1}}.
$$
By the definition of $(r_i)$ and $(\zeta_i)$, we have 
$$
r_{\ell}^{6k\gamma_{\ell+1}} r_{\ell-1}^{-\beta_{\ell+1}}\leq r_{\ell-1}^{\frac{\zeta_{\ell+1}}{12k}\cdot6k\gamma_{\ell+1}-\beta_{\ell+1}}\leq r_{\ell-1}^{\frac{-\beta_{\ell+1}}{2}}.
$$
It remains to notice that $r_{\ell-1}=t_k^{\theta_\ell}$ (where $\theta_\ell$ does not depend on $(t_i)$ but only on $(\zeta_i)$). To deduce that for $\beta_{k+1}:= \theta_\ell \frac{\beta_{\ell+1}}{2}$, we can bound \eqref{boun:int} by 
$$
C(\prod_{i=0}^k \|f_i\|_6 )^{6\gamma_{\ell+1}}\frac{1}{t_k^{\beta_{k+1}}}\leq C(\prod_{i=0}^k \|f_i\|_6 )^{6\gamma_{\ell+1}}\frac{1}{\min_{0\leq i\leq k}(t_{i+1}-t_i)^{\beta_{k+1}}}.
$$
This finishes the proof of the $Q(k+1)$ property in this case.

\textbf{Case B.} $\ell=k$. Recall that in this case \eqref{eq:stopk} holds.
Fix
$$
0<\sigma:=\frac{1}{|t_k|^{\alpha^2}}<1, \text{\ \ \ where\ \ \ }  \alpha=\alpha_k :=\min \left(\frac{1}{3k}, \frac{\xi_k}{2}\right).
$$
We will assume that $t_1 \geq 1$, otherwise the result is immediate. We have
$$
\int_M f_0 \cdot \prod_{i=1}^kf_i \circ h^{\tau}_{t_i} \diff \mu^{\tau} = \int_M (\tau f_0) \cdot \prod_{i=1}^kf_i \circ h^{\tau}_{t_i} \diff \mu = \Big\langle \tau f_0, \prod_{i=1}^kf_i \circ h^{\tau}_{t_i}\Big\rangle_{L^2(M,\mu)}.
$$
By invariance of the Haar measure $\mu$ by the geodesic flow and by integration by parts, for $\sigma>0$, we can write
\begin{equation*}
\begin{split}
&\Big\langle \tau f_0, \prod_{i=1}^kf_i \circ h^{\tau}_{t_i}\Big\rangle = \frac{1}{\sigma} \int_0^{\sigma} \Big\langle (\tau f_0) \circ g_s, \prod_{i=1}^kf_i \circ h^{\tau}_{t_i} \circ g_s \Big\rangle \diff s \\
& = \frac{1}{\sigma}  \Big\langle (\tau f_0) \circ g_{\sigma},  \int_0^{\sigma} \prod_{i=1}^kf_i \circ h^{\tau}_{t_i} \circ g_s \diff s \Big\rangle - \frac{1}{\sigma} \int_0^{\sigma} \Big\langle X(\tau f_0) \circ g_s, \int_0^s \prod_{i=1}^kf_i \circ h^{\tau}_{t_i} \circ g_r \diff r \Big\rangle \diff s.
\end{split}
\end{equation*}
%\textcolor{blue}{Below we use the $L^2$ norms with $2$ different notations}
Cauchy-Schwarz inequality then yields
\begin{equation}\label{eq:thm2_CS}
\begin{split}
\left\lvert \int_M f_0 \cdot \prod_{i=1}^kf_i \circ h^{\tau}_{t_i} d\mu^{\tau} \right\rvert \leq & \frac{\norm{\tau f_0}_2}{\sigma} \norm{\int_0^{\sigma} \prod_{i=1}^kf_i \circ h^{\tau}_{t_i} \circ g_s \diff s}_2 \\
& + \frac{1}{\sigma} \int_0^{\sigma} \norm{X(\tau f_0)}_2 \norm{\int_0^s \prod_{i=1}^kf_i \circ h^{\tau}_{t_i} \circ g_r \diff r}_2 \diff s \\
\leq & \left( \frac{\norm{\tau f_0}_2}{\sigma} + \norm{X(\tau f_0)}_2 \right) \sup_{s \in [0,\sigma]} \norm{\int_0^s \prod_{i=1}^kf_i \circ h^{\tau}_{t_i} \circ g_r \diff r}_2\\
\leq &\frac{3}{\sigma} \norm{\tau}_6 \norm{f_0}_6\sup_{s \in [0,\sigma]} \norm{\int_0^s \prod_{i=1}^kf_i \circ h^{\tau}_{t_i} \circ g_r \diff r}_2.
\end{split}
\end{equation}
By the commutation relation in Lemma \ref{lem:erren}, we have
$$
\norm{\int_0^s \prod_{i=1}^kf_i \circ h^{\tau}_{t_i} \circ g_r \diff r}_2 = \norm{\int_0^s \prod_{i=1}^k f_i \circ g_r \circ h^{\tau}_{e^{r}t_i + A(x,r,t_i)} \diff r}_2.
$$
Since, for every $r \in [0,s]$ and $x \in M$, there  exists a constant $C'$ such that 
%\textcolor{blue}{is the below correct, i.e. don't we also need a global constant and norms of $f_i$?}
$$
\left\lvert \prod_{i=1}^k f_i \circ g_r ( h^{\tau}_{e^{r}t_i + A(x,r,t_i)} x)- \prod_{i=1}^k f_i (h^{\tau}_{e^{r}t_i + A(x,r,t_i)}x) \right\rvert \leq C' s \prod_{i=1}^k \norm{f_i}_{6},
$$
it follows
\begin{equation}\label{eq:thm2_step1}
 \norm{\int_0^s \prod_{i=1}^k f_i \circ g_r \circ h^{\tau}_{e^{r}t_i + A(x,r,t_i)} \diff r - \int_0^s \prod_{i=1}^k f_i \circ h^{\tau}_{e^{r}t_i + A(x,r,t_i)} }_2 \leq C's^{2}\prod_{i=1}^k \norm{f_i}_{6}.
\end{equation}
Moreover, by Lemma \ref{lem:erren}, we have
\begin{equation}\label{eq:thm2_step2}
\begin{split}
& \norm{\int_0^s \prod_{i=1}^k f_i \circ h^{\tau}_{e^{r}t_i + A(x,r,t_i)} - \int_0^s \prod_{i=1}^k f_i \circ h^{\tau}_{e^{r}t_i} \diff r}_2 \\
& \leq \int_0^s \norm{ \prod_{i=1}^k f_i \circ h^{\tau}_{e^{r}t_i + A(x,r,t_i)} - \prod_{i=1}^k f_i \circ h^{\tau}_{e^{r}t_i} }_2 \diff r  \leq C \left(\prod_{i=1}^k \norm{f_i}_6 \right) s \max_{r\in[0,s]} r \|A(x,r,t_k)\|_\infty^k \\
& \leq C\left(\prod_{i=1}^k \norm{f_i}_6 \right)  s^{2} \log^k |t_k|,
\end{split}
\end{equation}
for some constant $C >0$.
From \eqref{eq:thm2_CS}, using \eqref{eq:thm2_step1} and \eqref{eq:thm2_step2}, we obtain
\begin{equation}\label{eq:thm2_bound2}
\begin{split}
& \left\lvert \int_X f_0 \cdot \prod_{i=1}^kf_i \circ h^{\tau}_{t_i} d\mu^{\tau} \right\rvert  \leq  \frac{3 \norm{\tau}_6\norm{f_0}_{6}}{\sigma} \sup_{s \in [0,\sigma]} \norm{\int_0^s \prod_{i=1}^k f_i \circ g_r \circ h^{\tau}_{e^rt_i + A(x,r,t_i)} \diff r}_2 \\
& \quad \leq C_{\tau,f}  \sigma \log^k |t_k| +  \frac{3 \norm{\tau}_6 \norm{f_0}_{6}}{\sigma}  \sup_{s \in [0,\sigma]} \norm{\int_0^s \prod_{i=1}^k f_i \circ h^{\tau}_{e^rt_i} \diff r}_2.
\end{split}
\end{equation}
where we have defined $C_{\tau, f} = 6 \norm{\tau}_6 \prod_{i=0}^k \norm{f_i}_{6}$.
We now bound the two terms in the right hand-side of \eqref{eq:thm2_bound2} separately.
For the first term, by the choice of $\sigma$ and by  \eqref{eq:stopk}, we have that for every $\varepsilon>0$
\begin{equation}\label{eq:thm2_term1}
C_{\tau, f} \sigma \log^k|t_k| = C \frac{ (\log|t_k|)^k}{|t_k|^{\alpha^2}} \leq C \frac{1}{|t_k|^{\alpha^2 - \varepsilon}} \leq C \frac{1}{(\min_{0\leq i <k} |t_{i+1}-t_i|)^{(\alpha^2 - \varepsilon)/\xi_k}}.
\end{equation}
We now bound the second term in \eqref{eq:thm2_bound2}.
Define $0 < K_i = t_i/t_k \leq 1$. For all $x\in M$, changing variable $u=e^rt_k$, and integrating by parts, 
\begin{equation*}
\begin{split}
\left\lvert \int_0^s \prod_{i=1}^k f_i \circ h^{\tau}_{e^rt_i} (x) \diff r\right\rvert &= \left\lvert \int_{t_k}^{e^st_k} \prod_{i=1}^k f_i \circ h^{\tau}_{K_i u} (x) \frac{\diff u}{u}\right\rvert \\
&\leq \frac{1}{|t_k|}  \left\lvert \int_{t_k}^{e^st_k} \prod_{i=1}^k f_i \circ h^{\tau}_{K_i r} (x) \diff r\right\rvert +  \left\lvert \int_{t_k}^{e^st_k} \frac{1}{r^2} \int_{t_k}^{r} \prod_{i=1}^k f_i \circ h^{\tau}_{K_i u} (x) \diff u \diff r \right\rvert.
% \sup_{r \in [0,s]} \frac{1}{t_k}  \left\lvert \int_{t_k}^{e^rt_k} \prod_{i=1}^k f_i \circ h^{\tau}_{K_i u} (x) \diff u\right\rvert.
\end{split}
\end{equation*}
Therefore, 
%\textcolor{blue}{How do we get the second term?} \textcolor{red}{after taking the $L^2$ norm, we bound the second term above by the $\sup_r$ of the $L^2$ norm times $\int_{t_k}^{e^st_k} |r^{-2}| \diff r = (e^s-1)/|t_k|$. Is this OK?.}
\begin{equation*}
\begin{split}
\norm{\int_0^s \prod_{i=1}^k f_i \circ h^{\tau}_{e^rt_i} \diff r}_2 & \leq \frac{1}{|t_k|} \norm{\int_{t_k}^{e^st_k} \prod_{i=1}^k f_i \circ h^{\tau}_{K_i r} \diff r}_2 +\frac{e^s-1}{|t_k|} \sup_{r\in[0,s]}  \norm{\int_{t_k}^{e^rt_k} \prod_{i=1}^k f_i \circ h^{\tau}_{K_i u} \diff u}_2 \\
& \leq  \sup_{r\in[0,s]} \frac{2}{|t_k|}  \norm{\int_{t_k}^{e^rt_k} \prod_{i=1}^k f_i \circ h^{\tau}_{K_i u} \diff u}_2,
\end{split}
\end{equation*}
hence
\begin{equation}\label{eq:thm2_bound3}
\frac{3 \norm{\tau}_6 \norm{f_0}_{6}}{\sigma}  \sup_{s \in [0,\sigma]} \norm{\int_0^s \prod_{i=1}^k f_i \circ h^{\tau}_{e^rt_i} \diff r}_2 \leq \frac{6 \norm{\tau}_6 \norm{f_0}_{6}}{\sigma |t_k|}  \sup_{s \in [0,\sigma]}  \norm{\int_{t_k}^{e^st_k} \prod_{i=1}^k f_i \circ h^{\tau}_{K_i u} \diff u}_2.
\end{equation}
Let $0 \leq s \leq \sigma$. 
If $0 \leq s \leq |t_k|^{-2\alpha^2}$,  
then, obviously,
\begin{equation}\label{eq:thm2_bound3a}
\frac{6 \norm{\tau}_6 \norm{f_0}_{6}}{\sigma |t_k|}  \norm{\int_{t_k}^{e^st_k} \prod_{i=1}^k f_i \circ h^{\tau}_{K_i u} \diff u}_2 \leq C \frac{e^s-1}{\sigma} \leq 2C\frac{s}{\sigma} \leq 2C |t_k|^{ - \alpha^2}.
\end{equation}
Recall we are assuming that $t_1 \geq 1$. If $|t_k|^{-2\alpha^2}< s \leq \sigma$, we have
$$
\frac{1}{(e^s-1)|t_k|} \leq \frac{1}{s|t_k|}\leq \frac{1}{|t_k|^{1-2\alpha}} \leq \left( \frac{t_1}{t_k}\right)^{1-2\alpha} = K_1^{1-2\alpha^2},
$$
hence the assumption of Proposition \ref{lem:l2} is satisfied with $\varepsilon = \frac{1}{1-2\alpha^2} < \frac{k+1}{k}$, $m=t_k$, and  $n=e^{s}t_k$.
Thus, by Proposition \ref{lem:l2}, we get
\begin{equation*}
\begin{split}
& \frac{6 \norm{\tau}_6 \norm{f_0}_{6}}{\sigma |t_k|}  \norm{\int_{t_k}^{e^st_k} \prod_{i=1}^k f_i \circ h^{\tau}_{K_i u} \diff u}_2 \\
& \qquad = \frac{6 \norm{\tau}_6 \norm{f_0}_{6}(e^{s}-1)}{\sigma} \norm{ \frac{1}{(e^{s}-1)t_k}\int_{t_k}^{e^{s}t_k} \prod_{i=1}^k f_i \circ h^{\tau}_{K_iu}  \diff u}_2\\
& \qquad \leq C_{\tau,f} \frac{e^s-1}{\sigma} \frac{1}{(\min_{1\leq i < k} (K_{i+1}-K_i) (e^s-1)|t_k|)^{\delta}} \\
& \qquad \leq 2C_{\tau, f} \frac{s^{1-\delta}}{\sigma}  \frac{1}{(\min_{1\leq i < k} |t_{i+1}-t_i|)^{\delta}}\leq 2C_{\tau,f} \frac{1}{(\sigma \min_{1\leq i < k} |t_{i+1}-t_i|)^{\delta}}.
\end{split}
\end{equation*}
By \eqref{eq:thm2_max}, the maximum above can be taken within $0\leq i < k$. Moreover, by \eqref{eq:stopk} and the definition of $\sigma$,  
\begin{equation*}
\begin{split}
\frac{1}{\sigma \min_{1\leq i < k} |t_{i+1}-t_i|} &= \frac{1}{\sigma \min_{0\leq i < k} |t_{i+1}-t_i|} =\frac{|t_k|^{\alpha^2}}{ \min_{0\leq i < k} |t_{i+1}-t_i|} \\
& \leq \frac{1}{\min_{0\leq i < k} |t_{i+1}-t_i|^{1-\alpha^2/\xi_k}},
\end{split}
\end{equation*}
and, by definition, $1-\alpha^2/\xi_k>0$.
Thus we obtain 
\begin{equation}\label{eq:thm2_bound3b}
\frac{6 \norm{\tau}_6 \norm{f_0}_{6}}{\sigma |t_k|} \norm{\int_{t_k}^{e^st_k} \prod_{i=1}^k f_i \circ h^{\tau}_{K_i u} \diff u}_2  \leq 2C_{\tau,f} \frac{1}{(\min_{0\leq i < k} |t_{i+1}-t_i|)^{\delta(1-\alpha^2/\xi_k)}}.
\end{equation}
In both cases $0 \leq s \leq |t_k|^{-2\alpha}$ or $|t_k|^{-2\alpha} < s \leq \sigma$, by \eqref{eq:thm2_bound3a} and \eqref{eq:thm2_bound3b}, we deduce that 
$$
\sup_{s\in[0,\sigma]}\frac{6 \norm{\tau}_6 \norm{f_0}_{6}}{\sigma |t_k|}  \norm{\int_{t_k}^{e^rt_k} \prod_{i=1}^k f_i \circ h^{\tau}_{K_i u}  \diff u}_2 \leq C \max(1,C_{\tau,f}) \frac{1}{(\min_{0\leq i < k} |t_{i+1}-t_i|)^{\delta}},
$$
for some $\delta >0$, so that, by \eqref{eq:thm2_bound3},
\begin{equation}\label{eq:thm2_term2}
\begin{split}
\frac{3 \norm{\tau}_6 \norm{f_0}_{6}}{\sigma} \sup_{s \in [0,\sigma]} \norm{\int_0^s \prod_{i=1}^k f_i \circ h^{\tau}_{e^rt_i} \diff r}_2 \leq C \max(1,C_{\tau,f}) \frac{1}{(\min_{0\leq i < k} |t_{i+1}-t_i|)^{\delta}}.
\end{split}
\end{equation}
The claim then follows by \eqref{eq:thm2_term1} and \eqref{eq:thm2_term2}.
\end{proof}

\subsection*{Acknowledgements}
We would like to thank Giovanni Forni for several discussions. We are also grateful to the organizers of the conference \lq\lq Dynamics of Parabolic Flows\rq\rq\ held at the University of Z\"urich in July 2019 for the opportunity to further discuss the project.
We thank the referee for his/her careful reading and corrections.

%%%%%%%%%%

\end{document}